\documentclass{amsart}

\usepackage{mathrsfs}

\allowdisplaybreaks[4]

\usepackage{amssymb, amsmath}
\usepackage{cases}
\usepackage{latexsym}
\usepackage{amsmath}
\usepackage[arrow,matrix]{xy}
\usepackage{stmaryrd}
\usepackage{amsfonts}
\usepackage{amsmath,amssymb,amscd,bbm,amsthm,mathrsfs,dsfont}
\usepackage{fancyhdr}
\usepackage{amsxtra,ifthen}
\usepackage{verbatim}

\numberwithin{equation}{section}

\theoremstyle{plain}
\newtheorem{theorem}{Theorem}[section]

\newtheorem{corollary}[theorem]{Corollary}

\theoremstyle{definition}

\newtheorem{remark}[theorem]{Remark}

\newtheorem{question}[theorem]{Question}
\newtheorem{claim}[theorem]{Claim}

\usepackage{cases}

\begin{document}

\title[Normal elements of completed group algebras over ${\rm SL}_3(\mathbb{Z}_p) $]
{Normal elements of completed group algebras over ${\rm SL}_3(\mathbb{Z}_p) $}

\author{Dong Han and Feng Wei}

\address{Han: School of Mathematics and Information Science, Henan Polytechnic University, Jiaozuo, 454000, P. R. China}

\email{lishe@hpu.edu.cn}

\address{Wei: School of Mathematics and Statistics, Beijing Institute of Technology,
Beijing, 100081, P. R. China}

\email{daoshuo@hotmail.com}\email{daoshuo@bit.edu.cn}

\begin{abstract}
Let $p$ be a prime integer and $\mathbb{Z}_p$ be the ring of
$p$-adic integers. By a purely computational approach we prove that each nonzero normal
element of a completed group algebra over the special linear group
${\rm SL}_3(\mathbb{Z}_p)$ is a unit.
This give a positive answer to an open question in \cite{WeiBian2} and
make up for an earlier mistake in \cite{WeiBian1} simultaneously.
\end{abstract}



\thanks{This work of the first author is supported by the Doctor Foundation of Henan Polytechnic University (B2010-21) and the Natural Science Research Program of Education Department of Henan Province (16A110031 and 15A110026).}

\date{\today}

\maketitle

\section{Introduction}
\label{xxsec1}

Let $p$ be a prime integer, and let $\mathbb{Z}_p$ denote the ring of
$p$-adic integers. A group $G$ is \textit{compact p-adic analytic}
if it is a topological group which has the structure of a $p$-adic
analytic manifold - that is, it has an atlas of open subsets of
$\mathbb{Z}^n_p$ , for some $n\geq 0$. Such groups can be
characterized in a more intrinsic way. A topological group $G$ is
compact $p$-adic analytic if and only if $G$ is a closed subgroup of
the general linear group ${\rm GL}_n(\mathbb{Z}_p)$ for some $n\geq
1$. In this paper we will consider the so-called \textit{completed
group algebras} of $G$
$$
\Lambda_G :=\varprojlim_{N\unlhd G} \mathbb{Z}_p[G/N],
$$
where the inverse limit is taken over the open normal subgroups $N$
of $G$. Closely related to $\Lambda_G$ is its epimorphic image
$\Omega_G$, which is defined as
$$
\Omega_G : =\varprojlim_{N\unlhd G} \mathbb{F}_p[G/N],
$$
where $\mathbb{F}_p$ is the finite field of $p$ elements. These algebras
with topological setting were defined and studied by Lazard in
his seminal paper \cite{Lazard} at first. They are complete semilocal
noetherian rings, which are in general noncommutative. Under the
name of Iwasawa algebras, these algebras are well-established and
have an increasing interest to number theorists, because of their
connections with number theory and arithmetic algebraic geometry.
On the other hand, it seems that explicit description, by generators and relations, of these algebras themselves
and its ideals were inaccessible. However, Serre's presentation of semi-simple algebras and Steinberg's
presentation of Chevalley groups \cite{Serre, Steinberg} make us believe that the objects
coming from semi-simple split groups have explicit presentation.
Indeed, for any odd prime $p$, Clozel in his paper \cite{Clozel1} gives explicit presentations for
the afore-mentioned two completed group algebra over
the first congruence subgroup of ${\rm SL}_2(\mathbb{Z}_p)$, which is
$\Gamma_1({\rm SL}_2(\mathbb{Z}_p))={\rm ker}({\rm SL}_2(\mathbb{Z}_p)\longrightarrow {\rm SL}_2(\mathbb{F}_p))$.
More recently, Ray \cite{Ray1, Ray2}  extended Clozel's work to the cases of semi-simple, simply
connected Chevalley groups over $\mathbb{Z}_p$  and pro-$p$ Iwahori subgroups of ${\rm GL}_n(\mathbb{Z}_p)$.

For completed group algebras or general noetherian algebras, we quite often focus on its two-sided ideals, especially its prime
ideals. Unfortunately, no much more information is provided with the ideal structure of noncommutative completed group algebras. Although
we have noted that central elements of $G$ and closed normal subgroups give rise to ideals, the lack of examples with respect to
ideals make us embarrass and is the most pressing problem in this topic. One natural question is:  is there a
mechanism for constructing ideals of completed group algebras which involves neither central elements nor closed normal subgroups ?
Recall that a uniform pro-$p$ group $G$ is \textit{almost simple} provided its Lie algebra has no non-trivial ideals . This is equivalent
to saying that every non-trivial closed normal subgroup of $G$ is open.  In \cite{Harris},
M. Harris claimed that, for an almost simple uniform pro-$p$ group $G$, any closed subgroup $H$ of $G$ with $2\, {\rm dim}\, H > {\rm dim}\, G$
gives rise to a non-zero two-sided ideal in $\Omega_G$, namely the annihilator of the ``Verma module"
constructed by induction from the simple $\Omega_H$-module. Unfortunately, Jordan Ellenberg observed that
the proof of the main theorem of \cite{Harris} contains a gap. We remind
the reader that $r\in \Omega_G$ is \textit{normal} if $r\Omega_G=\Omega_G r$. Another closely related question is:
For an almost simple uniform pro-$p$ group $G$, with $G\ncong \mathbb{Z}_p$, must any
nonzero normal element of $\Omega_G$ be a unit?
It was well-known that normal elements of associative algebras are
closely related to their ideals, especially their reflexive ideals. By a purely computational approach,
we prove that each nonzero normal element of the completed group algebra $\Omega_G$ over $\Gamma_1({\rm SL}_2(\mathbb{Z}_p))$
is a unit, see \cite[Theorem 9]{WeiBian1}. It is natural to describe the normal elements of the completed group
algebras over the special linear groups ${\rm SL}_3(\mathbb{Z}_p)$ and ${\rm SL}_n(\mathbb{Z}_p)$.
It is so pity that  the proof of \cite[Theorem 9]{WeiBian1} works at this point only for $G=\Gamma_1({\rm SL}_2(\mathbb{Z}_p))$
and for $\Omega_G$. We are sincerely grateful to Professor Dan Segal and Professor Stuart Mrgolis for drawing our
attention to an error in \cite{WeiBian1}. They inform us that similar statements for the completed group algebras of
the first congruence subgroups $\Gamma_1({\rm SL}_3(\mathbb{Z}_p))$ and $\Gamma_1({\rm SL}_n(\mathbb{Z}_p))$ can not be achieved by
analogous proofs of \cite[Theorem 9]{WeiBian1}.  In this situation, we must change the two statements---
\cite[Theorems 13 and 14]{WeiBian1}---into two open questions in \cite{WeiBian2}.

The purpose of this paper is to describe the normal elements of completed group algebras over the special
linear groups ${\rm SL}_3(\mathbb{Z}_p)$ by a purely computational method. Although we utilize some ideas of \cite{WeiBian1},
the adopted computational method in the current work is rather different from the original one.
We adjust and modify the original computational method considerably, see Claim 10 and Claim 11
of \cite{WeiBian1}, and Claim \ref{xxsec4.c1} and \ref{xxsec4.c5} of the current work. It
turns out that that each nonzero normal element of the completed group algebra $\Omega_G$ over
the first congruence subgroup $G=\Gamma_1({\rm SL}_3(\mathbb{Z}_p))$ is a unit.
This explicitly give a positive answer to the open question in \cite[Question 0.1]{WeiBian2} and
also make up for an earlier mistake in \cite[Theorem 13]{WeiBian1}.

The organization of this paper is as follows. After Introduction, we first recall some basic facts concerning
$p$-adic analytic groups ${\rm SL}_n(\mathbb{Z}_p)$ and its completed group algebras $\Omega_G$
in the Preliminaries. Section \ref{xxsec3} is contributed to complicated
computations of Lie brackets of topological generators of the completed group algebra $\Omega_G$ over
$G={\rm SL}_3(\mathbb{Z}_p)$. The proof of our main theorem (Theorem \ref{xxsec4.t1})
is given in Section \ref{xxsec4}. Some potential topics for further research are proposed in the last section.

\section{Preliminaries}
\label{xxsec2}

Let $n, t$ be positive integers. The \textit{$t$-th congruence subgroup in ${\rm SL}_n({\mathbb{Z}}_p)$}
is the kernel of the canonical epimorphism from ${\rm SL}_n({\mathbb{Z}}_p)$ to ${\rm
SL}_n({\mathbb{Z}}_p/p^t{\mathbb{Z}}_p)$. As usual, we denote it by $\Gamma_t({\rm SL}_n({\mathbb{Z}}_p))$.
It is easy to verify that $\Gamma_t({\rm SL}_n({\mathbb{Z}}_p))$ is a compact $p$-adic
analytic group. In the current work, we mainly investigate the completed
group algebra $\Omega_G$ of the \textit{first congruence
subgroup} $G=\Gamma_1({\rm SL}_n({\mathbb{Z}}_p))$ in ${\rm
SL}_n({\mathbb{Z}}_p)$ . We can fix a topological generating set for
$G$ as follows:

(1) Type of upper triangular matrix
$$
x_{ij}=\left [
\begin{array}{ccccccc}
1 &  &  &  &  &  & \\
  & \ddots  &  &  &  &  & \\
  &  & 1  & \cdots &  p \\
  &  &  & \ddots & \vdots &  & \\
  &  &  &  & 1 &  & \\
  &  &  &  &  & \ddots & \\
  &  &  &  &  &  & 1\\
\end{array}\right ](i<j),
$$
where the entry of $x_{ij}$ in the $i$-th row and $j$-th column is
$p$.

(2) Type of diagonal matrix
$$
x_{iijj}= \left[
\begin{array}{cccccccccc}
1 &  &  &  &  &  & & & &\\
  & \ddots  &  &  &  &  & & & &\\
  &  & 1  & &  & & & & &\\
  &  &  & \ddots & &  & & & &\\
  &  &  &  & 1+p &  & & & &\\
  &  &  &  &  & (1+p)^{-1} & & & &\\
  &  &  &  &  &  & \ddots  & & &\\
  &  &  &  &  &  & &  1 & &\\
  &  &  &  &  &  & &  &  \ddots &\\
  &  &  &  &  &  & &  & & 1\\
\end{array}\right](i=j-1),
$$
where the entry of $x_{iijj}$ in the $i$-th row and $i$-th column is
$1+p$ and the entry of $x_{iijj}$ in the $j$-th row and $j$-th
column is $(1+p)^{-1}$.

 (3) Type of lower triangular matrix
$$
x_{ij}=\left[
\begin{array}{ccccccc}
1 &  &  &  &  &  & \\
  & \ddots  &  &  &  &  & \\
  &  &  1  & & & & \\
  &  &  \vdots & \ddots  & &  & \\
  &  &  p &  \cdots & 1 &  & \\
  &  &  &  &  & \ddots & \\
  &  &  &  &  &  & 1\\
\end{array}\right](i>j),
$$
where the entry of $x_{ij}$ in the $i$-th row and $j$-th column is
$p$.

It is not difficult to verify that the number of topological generators for $G=\Gamma_1({\rm SL}_n({\mathbb{Z}}_p))$ is ${n^2-1}$. When
certain complicated computations are involved, the type and number of topological generators will be useful. It follows from the
discussion of \cite[\S 7.1]{DixonSautoyMannSegal} that the ordinary group algebra
${\mathbb{F}}_p[G]$ can embed into $\Omega_G$. For $i=1,2, \cdots n,
j=1, 2,\cdots, n$, let us set
$$
y_{ij}=x_{ij}-1(i< j), \ \ \ y_{iijj}=x_{iijj}-1(i=j-1), \ \ \ y_{ij}=x_{ij}-1(i> j),
$$
then $y_{ij}(i<j), y_{iijj}(i=j-1), y_{ij}(i>j)\in
{\mathbb{F}}_p[G]\subseteq \Omega_G$. Thus we can produce various
monomials in the $y_{ij}(i<j), y_{iijj}(i=j-1), y_{ij}(i>j)$: if
$\alpha=(\alpha_{12}, \cdots, \alpha_{1n}, \alpha_{23}, \cdots,
\alpha_{2n}, \cdots, \alpha_{(n-1)n}, \alpha_{1122}, \cdots, $
$\alpha_{(n-1)(n-1)nn}, \alpha_{21}, \alpha_{31}, \alpha_{32},
\cdots, \\ \alpha_{n1}, \cdots, \alpha_{n(n-1)})$ is a $(n^2$
$-1)$-tuple of nonnegative integers, we define
\begin{eqnarray}
&{\rm {\bf{y}}}^\alpha=y_{12}^{\alpha_{12}} \cdots
y_{1n}^{\alpha_{1n}}y_{23}^{\alpha_{23}} \cdots y_{2n}^{\alpha_{2n}}
\cdots y_{(n-1)n}^{\alpha_{(n-1)n}} y_{1122}^{\alpha_{1122}} \cdots\nonumber\\
&{\phantom{aaaaaaaaaaaaa}} y_{(n-1)(n-1)nn}^{\alpha_{(n-1)(n-1)nn}}y_{21}^{\alpha_{21}} y_{31}^{\alpha_{31}}y_{32}^{\alpha_{32}} \cdots
y_{n1}^{\alpha_{n1}} \cdots y_{n(n-1)}^{\alpha_{n(n-1)}}\in
\Omega_G\nonumber.
\end{eqnarray}
It should be remarked that the expressions of these monomials depend
on our choice of ordering of the $y_{ij}$'s$(i<j)$,
$y_{iijj}$'s$(i=j-1)$ , $y_{ij}$'s$(i>j)$, because $\Omega_G$ is
noncommutative unless $G$ is abelian. The following result shows
that $\Omega_G$ is a ``noncommutative formal power series ring".

\begin{theorem}\label{xxsec0.1}{\rm \cite[Theorem 7.23]{DixonSautoyMannSegal}}
Every element $r$ of $\Omega_G$ is equal
to the sum of a uniquely determined convergent series
$$
r=\sum_{\alpha \in {\mathbb{N}}^{n^2-1}} r_\alpha {\rm
{\bf{y}}}^\alpha,
$$
where $r_\alpha\in {\mathbb{F}}_p$ for all $\alpha\in
{\mathbb{N}}^{n^2-1}$.
\end{theorem}

As a direct consequence of this result we have

\begin{corollary}\label{xxsec0.2}
The Jacobson radical $J$ of $\Omega_G$ is equal to
\begin{eqnarray}
&J=y_{12} \Omega_G+ \cdots + y_{1n}\Omega_G +y_{21} \Omega_G +
\cdots+ y_{2n}\Omega_G+ \cdots+y_{(n-1)n}\Omega_G\nonumber\\
&{\phantom{aaaaaaaaa}} +y_{1122}\Omega_G+\cdots+y_{(n-1)(n-1)nn}\Omega_G+y_{21}\Omega_G
+y_{31}\Omega_G+y_{32}\Omega_G+\cdots\nonumber\\
&
+y_{n1}\Omega_G+\cdots+y_{n(n-1)}\Omega_G\nonumber\\
&
\hspace{10pt}=\Omega_Gy_{12}+ \cdots + \Omega_Gy_{1n}
+\Omega_Gy_{21}+ \cdots+ \Omega_Gy_{2n}+ \cdots+\Omega_Gy_{(n-1)n}\nonumber\\
&{\phantom{aaaaaaaaa}}+\Omega_Gy_{1122}+\cdots+\Omega_Gy_{(n-1)(n-1)nn}+\Omega_G y_{21}
+\Omega_Gy_{31}+\Omega_Gy_{32}+\cdots\nonumber\\
&
+\Omega_Gy_{n1}+\cdots+\Omega_Gy_{n(n-1)}.\nonumber
\end{eqnarray}
Moreover, $\Omega_G/J\cong {\mathbb{F}}_p$.
\end{corollary}

Theorem \ref{xxsec0.1} implies that the monomials $\{{\rm
{\bf{y}}}^\alpha: \alpha\in {\mathbb{N}}^{n^2-1}\}$ form a
topological basis for $\Omega_G$ and is thus analogous to the
classical Poincar$\acute{\rm e}$-Birkhoff-Witt theorem for Lie
algebras $\mathfrak{g}$ over a field $k$ which gives a vector space
basis for the enveloping algebra ${\mathcal{U}}(\mathfrak{g})$ in
terms of monomials in a fixed basis for $\mathfrak{g}$ \cite{Dixmier}.
Some explicit computations in $\Omega_G$ are much more difficult
than those in ${\mathcal{U}}(\mathfrak{g})$, which will be seen in
the sequel.

\section{  Lie Brackets of Generators of the Completed Group Algebra}\label{xxsec3}

We shall consider the normal elements of the completed group algebra $\Omega_G$  with $G =\Gamma_1({\rm SL}_3(\mathbb{Z}_p))$.
For this we need to discuss the Lie bracket of generators for the ordinary group algebra $\mathbb{F}_p[G]$.
Although part of them have been presented in \cite{WeiBian1} , it is indispensable for our later discussion.
Now we briefly sketch the relevant contents for the convenience of the reader.

\begin{theorem}\label{xxsec2.t1}
Let p be an odd prime number and
$$
\begin{aligned}
 &x_{12}=\left[
        \begin{array}{ccc}
          1 & p & 0 \\
          0 & 1 & 0 \\
          0 & 0 & 1 \\
        \end{array}
      \right]
 ,x_{13}=\left[
        \begin{array}{ccc}
          1 & 0 & p \\
          0 & 1 & 0 \\
          0 & 0 & 1 \\
        \end{array}
      \right],x_{23}=\left[
        \begin{array}{ccc}
          1 & 0 & 0 \\
          0 & 1 & p \\
          0 & 0 & 1 \\
        \end{array}
      \right],\\
& x_{1122}=\left[
        \begin{array}{ccc}
          1+p & 0 & 0 \\
          0 & (1+p)^{-1} & 0 \\
          0 & 0 & 1 \\
        \end{array}
      \right]
 ,x_{2233}=\left[
        \begin{array}{ccc}
          1 & 0 & 0 \\
          0 & 1+p & 0 \\
          0 & 0 & (1+p)^{-1} \\
        \end{array}
      \right],\\
      &x_{21}=\left[
        \begin{array}{ccc}
          1 & 0 & 0 \\
          p & 1 & 0 \\
          0 & 0 & 1 \\
        \end{array}
      \right]
 ,x_{31}=\left[
        \begin{array}{ccc}
          1 & 0 & 0 \\
          0 & 1 & 0 \\
          p & 0 & 1 \\
        \end{array}
      \right],x_{32}=\left[
        \begin{array}{ccc}
          1 & 0 & 0 \\
          0 & 1 & 0 \\
          0 & p & 1 \\
        \end{array}
      \right]
\end{aligned}
$$
be a topological generating set for  $G =\Gamma_1({\rm SL}_3(\mathbb{Z}_p))$ and $y_{ij} = x_{ij} - 1(i < j)$,
$y_{iijj} = x_{iijj} - 1(i = j - 1)$, $y_{ij}= x_{ij}- 1(i > j)$. Then for any nonnegative integers $r$ and $s$, we have
$$\textcircled{1}\ [y^{p^r}_{12}, y^{p^s}_{13}] = [y^{p^r}_{13}, y^{p^s}_{23}] = 0,$$
$$\textcircled{2}\  [y^{p^r}_{12}, y^{p^s}_{23}]= (1+y^{p^r}_{12} )[1 -  (1 + y^{p^{r+s+1}}_{13} )^{- 1}](1 + y^{p^s}_{23});$$
$$\textcircled{3}\  [y^{p^r}_{1122}, y^{p^s}_{2233}]= 0;$$
$$\textcircled{4}\  [y^{p^r}_{21}, y^{p^s}_{31}] = [y^{p^r}_{31}, y^{p^s}_{32}] = 0,$$
$$\textcircled{5}\  [y^{p^r}_{21}, y^{p^s}_{32}]= (1+y^{p^r}_{21} )[1 -  (1 + y^{p^{r+s+1}}_{31} )](1 + y^{p^s}_{32});$$
$$\textcircled{6}\  [y^{p^r}_{12}, y^{p^s}_{1122}]= (1+y^{p^r}_{12} )[1 -  (1 + y^{p^r}_{12} )^{(1+p)^{2p^s}-1}](1 + y^{p^s}_{1122});$$
$$\textcircled{7}\  [y^{p^r}_{13}, y^{p^s}_{1122}]= (1+y^{p^r}_{13} )[1 -  (1 + y^{p^r}_{13} )^{(1+p)^{p^s}-1}](1 + y^{p^s}_{1122});$$
$$\textcircled{8}\  [y^{p^r}_{23}, y^{p^s}_{1122}]= (1+y^{p^r}_{23} )[1 -  (1 + y^{p^r}_{23} )^{(1+p)^{-p^s}-1}](1 + y^{p^s}_{1122});$$
$$\textcircled{9}\  [y^{p^r}_{12}, y^{p^s}_{2233}]= (1+y^{p^r}_{12} )[1 -  (1 + y^{p^r}_{12} )^{(1+p)^{-p^s}-1}](1 + y^{p^s}_{2233});$$
$$\textcircled{\footnotesize{10}}\  [y^{p^r}_{13}, y^{p^s}_{2233}]= (1+y^{p^r}_{13} )[1 -  (1 + y^{p^r}_{13} )^{(1+p)^{p^s}-1}](1 + y^{p^s}_{2233});$$
$$\textcircled{\footnotesize{11}}\  [y^{p^r}_{23}, y^{p^s}_{2233}]= (1+y^{p^r}_{23} )[1 -  (1 + y^{p^r}_{23} )^{(1+p)^{2p^s}-1}](1 + y^{p^s}_{2233});$$
$$
\begin{aligned}
& \textcircled{\footnotesize{12}}\  [y^{p^r}_{12}, y^{p^s}_{21}]=(1+y^{p^r}_{12} )[1- (1 + y^{p^{2r+s+2}}_{12} )^{-(1+p^{r+s+2})^{-1}}(1 + y_{1122})^{\beta}\\
 &\ \ \ \ \ \ \ \ \ \ \ \ \ \ \ \ \ \times(1 +y^{p^{r+2s+2}}_{21} )^{-(1+p^{r+s+2})^{-1}}](1 + y^{p^s}_{21}),
\end{aligned}
$$
where $\beta=\sum_{k=0}^{\infty} \beta_kp^k\in \mathbb{Z}_p$, and $\beta_k\in \mathbb{Z} $, $0\leq\beta_k\leq p-1$ satisfying
$ \beta_0 = \beta_1=\cdots = \beta_{r+s}=0$, $\beta_{r+s+1}=p-1$, $\beta_{r+s+2}=(\frac{p^{r+s+1}-1}{2}-1)\mod p$, $\cdots ;$

$$\textcircled{\footnotesize{13}}\  [y^{p^r}_{13}, y^{p^s}_{21}]= (1+y^{p^r}_{21} )[1 -  (1 + y^{p^{r+s+1}}_{23} ) ](1 + y^{p^s}_{21});\ \ [y^{p^r}_{23}, y^{p^s}_{21}]=0;$$
$$\textcircled{\footnotesize{14}}\  [y^{p^r}_{12}, y^{p^s}_{31}]= (1+y^{p^r}_{12} )[1 -  (1 + y^{p^{r+s+1}}_{32} ) ](1 + y^{p^s}_{31});$$
$$
\begin{aligned}
&\textcircled{\footnotesize{15}}\  [y^{p^r}_{13}, y^{p^s}_{31}]=(1+y^{p^r}_{13} )[1- (1 + y^{p^{2r+s+2}}_{13} )^{-(1+p^{r+s+2})^{-1}}(1 + y_{1122})^{\beta}(1 + y_{2233})^{\beta}\\
 &\ \ \ \ \ \ \ \ \ \ \ \ \ \ \ \ \ \times(1 +y^{p^{r+2s+2}}_{31} )^{-(1+p^{r+s+2})^{-1}}](1 + y^{p^s}_{31}),
\end{aligned}
$$
where $\beta$ is as above;
$$\textcircled{\footnotesize{16}}\  [y^{p^r}_{23}, y^{p^s}_{31}]= (1+y^{p^r}_{23} )[1 -  (1 + y^{p^{r+s+1}}_{21} )^{- 1}](1 + y^{p^s}_{31});\ \ [y^{p^r}_{12}, y^{p^s}_{32}]=0;$$
$$\textcircled{\footnotesize{17}}\  [y^{p^r}_{13}, y^{p^s}_{32}]= (1+y^{p^r}_{13} )[1 -  (1 + y^{p^{r+s+1}}_{12} )^{- 1}](1 + y^{p^s}_{32});$$
$$
\begin{aligned}
& \textcircled{\footnotesize{18}}\  [y^{p^r}_{23}, y^{p^s}_{32}]=(1+y^{p^r}_{23} )[1- (1 + y^{p^{2r+s+2}}_{23} )^{-(1+p^{r+s+2})^{-1}}(1 + y_{2233})^{\beta}\\
 &\ \ \ \ \ \ \ \ \ \ \ \ \ \ \ \ \ \times(1 +y^{p^{r+2s+2}}_{32} )^{-(1+p^{r+s+2})^{-1}}](1 + y^{p^s}_{32}),
\end{aligned}
$$
where $\beta$ is as above;
$$\textcircled{\footnotesize{19}}\  [y^{p^r}_{1122}, y^{p^s}_{21}]= (1+y^{p^r}_{1122} )[1 -  (1 + y^{p^s}_{21} )^{(1+p)^{2p^r-1}}](1 + y^{p^s}_{21});$$
$$\textcircled{\footnotesize{20}}\  [y^{p^r}_{1122}, y^{p^s}_{31}]= (1+y^{p^r}_{1122} )[1 -  (1 + y^{p^s}_{31} )^{(1+p)^{p^r}-1}](1 + y^{p^s}_{31});$$
$$\textcircled{\footnotesize{21}}\  [y^{p^r}_{1122}, y^{p^s}_{32}]= (1+y^{p^r}_{1122} )[1 -  (1 + y^{p^s}_{32} )^{(1+p)^{-p^r}-1}](1 + y^{p^s}_{32});$$
$$\textcircled{\footnotesize{22}}\  [y^{p^r}_{2233}, y^{p^s}_{21}]= (1+y^{p^r}_{2233} )[1 -  (1 + y^{p^s}_{21} )^{(1+p)^{-p^r}-1}](1 + y^{p^s}_{21});$$
$$\textcircled{\footnotesize{23}}\ [y^{p^r}_{2233}, y^{p^s}_{31}]= (1+y^{p^r}_{2233} )[1 -  (1 + y^{p^s}_{31} )^{(1+p)^{p^r}-1}](1 + y^{p^s}_{31});$$
$$\textcircled{\footnotesize{24}}\ [y^{p^r}_{2233}, y^{p^s}_{32}]= (1+y^{p^r}_{2233} )[1 -  (1 + y^{p^s}_{32} )^{(1+p)^{2p^r}-1}](1 + y^{p^s}_{32}).$$

\end{theorem}

\begin{proof} \textcircled{1} This is a trivial computation.

 \textcircled{2}  One can directly verify
\begin{equation}\label{xxsec2ei2e1}
[y^{p^r}_{12}, y^{p^s}_{23}]=[x^{p^r}_{12}, x^{p^s}_{23}]=x^{p^r}_{12}(1- x^{-p^r}_{12}x^{p^s}_{23}x^{p^r}_{12}x^{-p^s}_{23})x^{p^s}_{23}.
\end{equation}
Thus it suffices to determine $x^{-p^r}_{12}x^{p^s}_{23}x^{p^r}_{12}x^{-p^s}_{23}$.
\begin{equation}\label{xxsec2ei2e2}
\begin{aligned}
 &x^{-p^r}_{12}x^{p^s}_{23}x^{p^r}_{12}x^{-p^s}_{23} \\
&=\left[
        \begin{array}{ccc}
          1 & -p^{r+1} & 0 \\
          0 & 1 & 0 \\
          0 & 0 & 1 \\
        \end{array}
      \right] \left[\begin{array}{ccc}
          1 & 0 & 0\\
          0 & 1 & p^{s+1} \\
          0 & 0 & 1 \\
        \end{array}
      \right]\left[
        \begin{array}{ccc}
          1 & p^{r+1} & 0 \\
          0 & 1 & 0 \\
          0 & 0 & 1 \\
        \end{array}
      \right] \left[\begin{array}{ccc}
          1 & 0 & 0\\
          0 & 1 & -p^{s+1} \\
          0 & 0 & 1 \\
        \end{array}
      \right]\\
      &=\left[
        \begin{array}{ccc}
          1 & -p^{r+1} &-p^{r+s+2} \\
          0 & 1 & p^{s+1} \\
          0 & 0 & 1 \\
        \end{array}
      \right] \left[
        \begin{array}{ccc}
          1 & p^{r+1} & 0 \\
          0 & 1 & 0 \\
          0 & 0 & 1 \\
        \end{array}
      \right] \left[\begin{array}{ccc}
          1 & 0 & 0\\
          0 & 1 & -p^{s+1} \\
          0 & 0 & 1 \\
        \end{array}
      \right]\\
       &=\left[
        \begin{array}{ccc}
          1 & 0 &-p^{r+s+2} \\
          0 & 1 & p^{s+1} \\
          0 & 0 & 1 \\
        \end{array}
      \right]\left[\begin{array}{ccc}
          1 & 0 & 0\\
          0 & 1 & -p^{s+1} \\
          0 & 0 & 1 \\
        \end{array}
      \right]=\left[
        \begin{array}{ccc}
          1 & 0 &-p^{r+s+2} \\
          0 & 1 & 0 \\
          0 & 0 & 1 \\
        \end{array}
      \right]=x_{13}^{-p^{r+s+1}}.
\end{aligned}\end{equation}
Taking (\ref{xxsec2ei2e2}) into (\ref{xxsec2ei2e1}), we obtain
$$
[y^{p^r}_{12}, y^{p^s}_{23}]=(1+y_{12})^{p^r}[1-(1+y_{13})^{-p^{r+s+1}}](1+y_{23})^{p^s}.
$$

\textcircled{3}-\textcircled{4} They are straightforward to compute.

\textcircled{5} Applying the computational method of \textcircled{2} yields it.

\textcircled{6} Let us consider the relation:
\begin{equation}\label{xxsec2ei61}
[y^{p^r}_{12}, y^{p^s}_{1122}]=[x^{p^r}_{12}, x^{p^s}_{1122}]=x^{p^r}_{12}(1- x^{-p^r}_{12}x^{p^s}_{1122}x^{p^r}_{12}x^{-p^s}_{1122})x^{p^s}_{1122}.
\end{equation}
One can compute
\begin{equation}\label{xxsec2ei62}
\begin{aligned}
 &x^{-p^r}_{12}x^{p^s}_{1122}x^{p^r}_{12}x^{-p^s}_{1122} \\
&=\left[
        \begin{array}{ccc}
          1 & -p^{r+1}& 0 \\
         0 & 1 & 0 \\
          0 & 0 & 1 \\
        \end{array}
      \right]\left[\begin{array}{ccc}
          (1+p)^{p^s} & 0 & 0\\
          0 & (1+p)^{-p^s} & 0\\
          0 & 0 & 1 \\
        \end{array}
      \right] \\&\ \ \ \ \ \   \times\left[
        \begin{array}{ccc}
          1 & p^{r+1}& 0 \\
        0 & 1 & 0 \\
          0 & 0 & 1 \\
        \end{array}
      \right] \left[\begin{array}{ccc}
        (1+p)^{-p^s} & 0 & 0\\
          0 & (1+p)^{p^s} & 0\\
          0 & 0 & 1 \\
        \end{array}
      \right]\\
      &=\left[
        \begin{array}{ccc}
          (1+p)^{p^s} & -p^{r+1}(1+p)^{-p^s}& 0 \\
          0 & (1+p)^{-p^s} &0 \\
          0& 0 & 1 \\
        \end{array}
      \right] \left[
        \begin{array}{ccc}
          1 & p^{r+1}& 0 \\
        0 & 1 & 0 \\
          0 & 0 & 1 \\
        \end{array}
      \right] \left[\begin{array}{ccc}
        (1+p)^{-p^s} & 0 & 0\\
          0 & (1+p)^{p^s} & 0\\
          0 & 0 & 1 \\
        \end{array}
      \right]\\
       &=\left[
        \begin{array}{ccc}
           (1+p)^{p^s} &  p^{r+1}(1+p)^{p^s}-  p^{r+1}(1+p)^{-p^s}&0 \\
          0& (1+p)^{-p^s} &0 \\
          0 & 0  & 1 \\
        \end{array}
      \right] \left[\begin{array}{ccc}
        (1+p)^{-p^s} & 0 & 0\\
          0 & (1+p)^{p^s} & 0\\
          0 & 0 & 1 \\
        \end{array}
      \right]\\
 &=\left[
        \begin{array}{ccc}
          1 & p^{r+1}(1+p)^{2p^s}-  p^{r+1} &0 \\
          0 & 1 & 0 \\
          0 & 0& 1 \\
        \end{array}
      \right]=x_{12}^{p^{r}(1+p)^{2p^s}-  p^{r}}=(1+y_{12})^{p^{r}(1+p)^{2p^s}-  p^{r}}.
\end{aligned}\end{equation}
Combining (\ref{xxsec2ei61}) with  (\ref{xxsec2ei62}) gives
$$
[y^{p^r}_{12}, y^{p^s}_{1122}]=(1+y_{12})^{p^r}[1-(1+y_{12})^{p^{r}(1+p)^{2p^s}-  p^{r}}](1+y_{1122})^{p^s}.
$$

 \textcircled{7}-\textcircled{\footnotesize{11}} The computational methods of
 \textcircled{7}-\textcircled{\footnotesize{11}} are similar to that of \textcircled{6}.

\textcircled{\footnotesize{12}}  In light of the relation
\begin{equation}\label{xxsec2ei121}
[y^{p^r}_{12}, y^{p^s}_{21}]=[x^{p^r}_{12}, x^{p^s}_{21}]=x^{p^r}_{12}(1- x^{-p^r}_{12}x^{p^s}_{21}x^{p^r}_{12}x^{-p^s}_{21})x^{p^s}_{21},
\end{equation}
it is sufficient for us to compute  $x^{-p^r}_{12}x^{p^s}_{21}x^{p^r}_{12}x^{-p^s}_{21}$.
\begin{equation}\label{xxsec2ei122}
\begin{aligned}
 &x^{-p^r}_{12}x^{p^s}_{21}x^{p^r}_{12}x^{-p^s}_{21} \\
&=\left[
        \begin{array}{ccc}
          1 & -p^{r+1}& 0 \\
         0 & 1 & 0\\
          0 & 0 & 1 \\
        \end{array}
      \right] \left[\begin{array}{ccc}
        1 &  0  & 0 \\
         p^{s+1} & 1 & 0\\
          0 & 0 & 1 \\
        \end{array}
      \right]\left[
        \begin{array}{ccc}
          1 & p^{r+1}& 0 \\
        0 & 1 &0 \\
          0 & 0 & 1 \\
        \end{array}
      \right] \left[\begin{array}{ccc}
          1 &  0  & 0 \\
         -p^{s+1} & 1 & 0\\
          0 & 0 & 1 \\
        \end{array}
      \right]\\
      &=\left[
        \begin{array}{ccc}
         1 -p^{r+s+2} & -p^{r+1} & 0 \\
          p^{s+1} & 1 & 0\\
          0& 0 & 1 \\
        \end{array}
      \right] \left[
        \begin{array}{ccc}
            1 & p^{r+1}& 0 \\
        0 & 1 &0 \\
          0 & 0 & 1 \\
        \end{array}
      \right] \left[\begin{array}{ccc}
          1 &  0  & 0 \\
         -p^{s+1} & 1 & 0\\
          0 & 0 & 1 \\
        \end{array}
      \right]\\
       &=\left[
        \begin{array}{ccc}
             1 -p^{r+s+2}&   -p^{2r+s+3} &0 \\
          p^{s+1}&1 + p^{r+s+2} &0 \\
          0 & 0  &  1 \\
        \end{array}
      \right] \left[\begin{array}{ccc}
    1 &  0  & 0 \\
         -p^{s+1} & 1 & 0\\
          0 & 0 & 1 \\
        \end{array}
      \right]\\
 &=\left[
        \begin{array}{ccc}
          1 -p^{r+s+2}+p^{2(r+s+2)} & -p^{2r+s+3}&0 \\
          -p^{r+2s+3} &1+ p^{r+s+2}& 0 \\
          0 & 0 & 1 \\
        \end{array}
      \right].\end{aligned}\end{equation}
 Applying triangular decomposition formula to the matrix in (\ref{xxsec2ei122}) yields
  \begin{equation}\label{xxsec2ei123} \begin{aligned}&\left[
        \begin{array}{ccc}
          1 -p^{r+s+2}+p^{2(r+s+2)} & -p^{2r+s+3}&0 \\
          -p^{r+2s+3} &1 +p^{r+s+2}& 0 \\
          0 & 0 & 1 \\
        \end{array}
      \right]\\&=\left[
        \begin{array}{ccc}
          1  & -p^{2r+s+3}(  1 +p^{r+s+2})^{-1}&0 \\
         0 & 1 & 0 \\
          0 & 0 & 1 \\
        \end{array}
      \right]\\
       &\ \ \times\left[
        \begin{array}{ccc}
          1 -p^{r+s+2}+p^{2(r+s+2)}-p^{3(r+s+2)} (  1 +p^{r+s+2})^{-1} & 0&0 \\
         0 &  1 +p^{r+s+2} & 0 \\
          0 & 0 & 1 \\
        \end{array}
      \right]\\
       &\ \ \times\left[
        \begin{array}{ccc}
          1 & 0&0 \\
         -p^{r+2s+3} (  1 +p^{r+s+2})^{-1} &  1  & 0 \\
          0 & 0 & 1 \\
        \end{array}
      \right]\\
  & =(1+y_{12})^{-p^{2r+s+2}(  1 +p^{r+s+2})^{-1}}\left[
        \begin{array}{ccc}
          (  1 +p^{r+s+2})^{-1} & 0&0 \\
         0 &  1 +p^{r+s+2} & 0 \\
          0 & 0 & 1 \\
        \end{array}
      \right]\\&\ \ \times(1+y_{21})^{-p^{r+2s+2}(  1 +p^{r+s+2})^{-1}}.
      \end{aligned}\end{equation}
We should note that  $$(1+p^{r+s+2})^{-1}=1-p^{r+s+2}+p^{2(r+s+2)}-p^{3(r+s+2)}+p^{4(r+s+2)}+\cdots.$$
It follows from the properties of $p$-adic integers that there exists one element $\beta$ such that $$(1+p^{r+s+2})^{-1}=(1+p)^{\beta},$$
where $\beta = \beta_0 + \beta_1p + \beta_2p^2 + \cdots+ \beta_{r+s}p^{r+s}+ \beta_{r+s+1}p^{r+s+1}+ \cdots $, $\beta_k\in \mathbb{Z} $ and
$0 \leq \beta_k\leq (p - 1)$. According to the expansion formula of $(  1 +p^{r+s+2})^{-1}$, we can compute all $\beta_k$. For instance,
$ \beta_0 = \beta_1=\cdots = \beta_{r+s}=0$, $\beta_{r+s+1}=p-1$, $\beta_{r+s+2}=(\frac{p^{r+s+1}-1}{2}-1)\mod p$, $\cdots$.
Thus (\ref{xxsec2ei123}) can be rewritten as
\begin{equation}\label{xxsec2ei124}\begin{aligned}&\left[
        \begin{array}{ccc}
          1 -p^{r+s+2}+p^{2(r+s+2)} & -p^{2r+s+3}&0 \\
          -p^{r+2s+3} & 1+p^{r+s+2}& 0 \\
          0 & 0 & 1 \\
        \end{array}
      \right]\\
&(1+y_{12})^{-p^{2r+s+2}(  1 +p^{r+s+2})^{-1}}\left[
        \begin{array}{ccc}
          (  1 +p)^{\beta} & 0&0 \\
         0 &  (  1 +p)^{-\beta} & 0 \\
          0 & 0 & 1 \\
        \end{array}
      \right](1+y_{21})^{-p^{r+2s+2}(  1 +p^{r+s+2})^{-1}}
\\&=(1+y_{12})^{-p^{2r+s+2}(  1 +p^{r+s+2})^{-1}}(1+y_{1122})^{\beta}(1+y_{21})^{-p^{r+2s+2}(  1 +p^{r+s+2})^{-1}} .\end{aligned}\end{equation}
Consequently, we have
$$
\begin{aligned}
&[y^{p^r}_{12}, y^{p^s}_{21}]=(1+y_{12})^{p^r}[1- (1+y_{12})^{-p^{2r+s+2}(  1 +p^{r+s+2})^{-1}}(1+y_{1122})^{\beta}\\
&{\phantom{[y^{p^r}_{12}, y^{p^s}_{21}]=}}\times(1+y_{21})^{-p^{r+2s+2}(  1 +p^{r+s+2})^{-1}}](1+y_{21})^{p^s}.
\end{aligned}$$

\textcircled{\footnotesize{13}} Let us see the Lie bracket   $[y^{p^r}_{13}, y^{p^s}_{21}]$.
$$
[y^{p^r}_{13}, y^{p^s}_{21}]=[x^{p^r}_{13}, x^{p^s}_{21}]=x^{p^r}_{13}(1- x^{-p^r}_{13}x^{p^s}_{21}x^{p^r}_{13}x^{-p^s}_{21})x^{p^s}_{21}.
$$
Note that
$$
\begin{aligned}
 &x^{-p^r}_{13}x^{p^s}_{21}x^{p^r}_{13}x^{-p^s}_{21} \\
&=\left[
        \begin{array}{ccc}
          1 & 0&-p^{r+1} \\
          0& 1 & 0 \\
          0 & 0 & 1 \\
        \end{array}
      \right] \left[\begin{array}{ccc}
          1 & 0 & 0\\
          p^{s+1}& 1 & 0\\
          0 & 0 & 1 \\
        \end{array}
      \right]\left[
        \begin{array}{ccc}
          1 & 0 & p^{r+1} \\
             0 & 1 & 0 \\
          0 & 0 & 1 \\
        \end{array}
      \right] \left[\begin{array}{ccc}
          1 & 0 & 0\\
          -p^{s+1}& 1 &0 \\
          0 &0 & 1 \\
        \end{array}
      \right]\\
      &=\left[
        \begin{array}{ccc}
          1 & 0& -p^{r+1} \\
        p^{s+1} & 1 &0 \\
          0& 0& 1 \\
        \end{array}
      \right] \left[
        \begin{array}{ccc}
           1 & 0 & p^{r+1} \\
             0 & 1 & 0 \\
          0 & 0 & 1 \\
        \end{array}
      \right] \left[\begin{array}{ccc}
          1 & 0 & 0\\
          -p^{s+1}& 1 &0 \\
          0 &0 & 1 \\
        \end{array}
      \right]\\
       &=\left[
        \begin{array}{ccc}
          1 & 0 &0 \\
         p^{s+1}& 1 & p^{r+s+2}\\
        0 & 0& 1 \\
        \end{array}
      \right]\left[\begin{array}{ccc}
          1 & 0 & 0\\
          -p^{s+1} & 1 & 0\\
          0 & 0& 1 \\
        \end{array}
      \right]=\left[
        \begin{array}{ccc}
          1 & 0 &0 \\
          0 & 1 & p^{r+s+2}\\
          0 & 0& 1 \\
        \end{array}
      \right]=x_{23}^{p^{r+s+1}}.
\end{aligned}$$
We therefore have
$$
[y^{p^r}_{13}, y^{p^s}_{21}]=(1+y_{13})^{p^r}[1-(1+y_{23})^{p^{r+s+1}}](1+y_{21})^{p^s}.
$$

\textcircled{\footnotesize{14}}-\textcircled{\footnotesize{18}} The proofs can be safely left as exercises.

\textcircled{\footnotesize{19}}-\textcircled{\footnotesize{24}} Let us sketch the proof of \textcircled{\footnotesize{19}}, the rest follow in a similar fashion. Here again, the Lie bracket can  be written as
$$
[y^{p^r}_{1122}, y^{p^s}_{21}]=[x^{p^r}_{1122}, x^{p^s}_{21}]=x^{p^r}_{1122}(1- x^{-p^r}_{1122}x^{p^s}_{21}x^{p^r}_{1122}x^{-p^s}_{21})x^{p^s}_{21}.
$$
We compute
$$
\begin{aligned}
 &x^{-p^r}_{1122}x^{p^s}_{21}x^{p^r}_{1122}x^{-p^s}_{21} \\
&=\left[
        \begin{array}{ccc}
          (1+p)^{-p^r} & 0 & 0\\
          0 & (1+p)^{p^r} & 0\\
          0 & 0 & 1 \\
        \end{array}
      \right] \left[\begin{array}{ccc}
           1 &0 & 0 \\
         p^{s+1} & 1 & 0 \\
          0 & 0 & 1 \\
        \end{array}
      \right]\\&\ \ \ \ \times\left[
        \begin{array}{ccc}
         (1+p)^{p^r} & 0 & 0\\
          0 & (1+p)^{-p^r} & 0\\
          0 & 0 & 1 \\
        \end{array}
      \right] \left[\begin{array}{ccc}
          1 & 0 & 0 \\
      - p^{s+1} & 1 & 0 \\
          0 & 0 & 1 \\
        \end{array}
      \right]\\
      &=\left[
        \begin{array}{ccc}
          (1+p)^{-p^r} & 0 & 0 \\
        p^{s+1}(1+p)^{p^r}& (1+p)^{p^r} &0 \\
          0& 0 & 1 \\
        \end{array}
      \right] \left[
        \begin{array}{ccc}
            (1+p)^{p^r} & 0 & 0\\
          0 & (1+p)^{-p^r} & 0\\
          0 & 0 & 1 \\
        \end{array}
      \right] \left[\begin{array}{ccc}
          1 & 0 & 0 \\
      - p^{s+1} & 1 & 0 \\
          0 & 0 & 1 \\
        \end{array}
      \right]\\
       &=\left[
        \begin{array}{ccc}
           1 & 0&0 \\
           p^{s+1}(1+p)^{2p^r}& 1 &0 \\
          0 & 0  & 1 \\
        \end{array}
      \right] \left[\begin{array}{ccc}
                 1 & 0 & 0 \\
      - p^{s+1} & 1 & 0 \\
          0 & 0 & 1 \\
        \end{array}
      \right]\\
 &=\left[
        \begin{array}{ccc}
          1 & 0 &0 \\
         p^{s+1}(1+p)^{2p^r}-  p^{s+1} & 1 & 0 \\
          0 & 0& 1 \\
        \end{array}
      \right]=x_{21}^{p^{s}(1+p)^{2p^r}-  p^{s}}=(1+y_{21})^{p^{s}(1+p)^{2p^r}-  p^{s}}.
\end{aligned}$$
This shows that
$$
[y^{p^r}_{1122}, y^{p^s}_{21}]=(1+y_{1122})^{p^r}[1-(1+y_{21})^{p^{s}(1+p)^{2p^r}-  p^{s}}](1+y_{21})^{p^s}.
$$

\end{proof}

In fact, we shall only need the lowest degree terms of the expansions of the above-mentioned Lie
brackets; these can be easily deduced from  Theorem \ref{xxsec2.t1}, giving

\begin{equation}\label{xxsec3.e10}
\begin{aligned}
 &[y^{p^r}_{12},\    y^{p^s}_{13}]_{\circ} = [y^{p^r}_{13},\    y^{p^s}_{23}]_{\circ} = 0,\   [y^{p^r}_{12},\    y^{p^s}_{23}]_{\circ}=y^{p^{r+s+1}}_{13},\   [y^{p^r}_{1122},\    y^{p^s}_{2233}]_{\circ} =0,\   \\
& [y^{p^r}_{21},\    y^{p^s}_{31}]_{\circ} = [y^{p^r}_{31},\    y^{p^s}_{32}]_{\circ} = 0,\   [y^{p^r}_{21},\    y^{p^s}_{32}]_{\circ} =-y_{31}^{p^{r+s+1}},\   \\
&[y^{p^r}_{12},\    y^{p^s}_{1122}]_{\circ} =-2y^{p^{r+s+1}}_{12},\   [y^{p^r}_{13},\    y^{p^s}_{1122}]_{\circ} =-y^{p^{r+s+1}}_{13},\   [y^{p^r}_{23},\    y^{p^s}_{1122}]_{\circ} =y^{p^{r+s+1}}_{23},\   \\ &[y^{p^r}_{12},\    y^{p^s}_{2233}]_{\circ} =y^{p^{r+s+1}}_{12},\   [y^{p^r}_{13},\    y^{p^s}_{2233}]_{\circ} =-y^{p^{r+s+1}}_{13},\   [y^{p^r}_{23},\    y^{p^s}_{2233}]_{\circ} =-2y^{p^{r+s+1}}_{23},\   \\&[y^{p^r}_{12},\    y^{p^s}_{21}]_{\circ} =y^{p^{r+s+1}}_{1122},\   [y^{p^r}_{13},\    y^{p^s}_{21}]_{\circ} =-y^{p^{r+s+1}}_{23},\   [y^{p^r}_{23},\    y^{p^s}_{21}]_{\circ}=0,\   \\&[y^{p^r}_{12},\    y^{p^s}_{31}]_{\circ}=-y^{p^{r+s+1}}_{32},\   [y^{p^r}_{13},\    y^{p^s}_{31}]_{\circ}=y^{p^{r+s+1}}_{1122}+y^{p^{r+s+1}}_{2233},\   [y^{p^r}_{23},\    y^{p^s}_{31}]_{\circ}=y^{p^{r+s+1}}_{21},\   \\
&[y^{p^r}_{12},\    y^{p^s}_{32}]_{\circ}=0,\    [y^{p^r}_{13},\    y^{p^s}_{32}]_{\circ}=y^{p^{r+s+1}}_{12},\   [y^{p^r}_{23},\    y^{p^s}_{32}]_{\circ}=y^{p^{r+s+1}}_{2233},\   [y^{p^r}_{1122},\    y^{p^s}_{21}]_{\circ} =-2y^{p^{r+s+1}}_{21}\\
&[y^{p^r}_{1122},\    y^{p^s}_{31}]_{\circ} =-y^{p^{r+s+1}}_{31},\   [y^{p^r}_{1122},\    y^{p^s}_{32}]_{\circ}= y^{p^{r+s+1}}_{32},\   [y^{p^r}_{2233},\    y^{p^s}_{21}]_{\circ}= y^{p^{r+s+1}}_{21},\   \\&[y^{p^r}_{2233},\    y^{p^s}_{31}]_{\circ}= -y^{p^{r+s+1}}_{31},\    [y^{p^r}_{2233},\    y^{p^s}_{32}]_{\circ}= -2y^{p^{r+s+1}}_{32}.
\end{aligned}
\end{equation}
where $[y^{p^r}_{12}, y^{p^s}_{13}]_{\circ}, [y^{p^r}_{13}, y^{p^s}_{23}]_{\circ}, \cdots, [y^{p^r}_{2233}, y^{p^s}_{32}]_{\circ}$
denote the lowest degree terms of the expansions of $[y^{p^r}_{12}, y^{p^s}_{13}]$, $[y^{p^r}_{13}, y^{p^s}_{23}]$,
$\cdots$, $[y^{p^r}_{2233}, y^{p^s}_{32}]$, respectively.  Henceforth, a similar and completely compatible
notation will be used in the whole paper.

Let $G =\Gamma_1({\rm SL}_3(\mathbb{Z}_p))$ be the first congruence kernel of  ${\rm SL}_3(\mathbb{Z}_p)$.
For convenience the topological generating set $\{x_{12},x_{13},x_{23}$,$x_{1122},x_{2233}$,$x_{21}, x_{31}, x_{32}\}$ for $G$
is briefly denoted by $\{x_1,x_2,\cdots,x_8\}$, and the corresponding generators in the ordinary group algebra $\mathbb{F}_p[G]$
are set  $y_k=x_k-1, k=1,2,\cdots, 8$. Now relations in (\ref{xxsec3.e10}) can be rewritten as

\begin{equation}\label{xxsec3.e11}
\begin{aligned}
 &[y^{p^r}_{1},\    y^{p^s}_{2}]_{\circ} = [y^{p^r}_{2},\    y^{p^s}_{3}]_{\circ} = 0,\   [y^{p^r}_{1},\    y^{p^s}_{3}]_{\circ}=y^{p^{r+s+1}}_{2},\   [y^{p^r}_{4},\    y^{p^s}_{5}]_{\circ} =0,\   \\
& [y^{p^r}_{6},\    y^{p^s}_{7}]_{\circ} = [y^{p^r}_{7},\    y^{p^s}_{8}]_{\circ} = 0,\   [y^{p^r}_{6},\    y^{p^s}_{8}]_{\circ} =-y_{7}^{p^{r+s+1}},\   \\
&[y^{p^r}_{1},\    y^{p^s}_{4}]_{\circ} =-2y^{p^{r+s+1}}_{1},\   [y^{p^r}_{2},\    y^{p^s}_{4}]_{\circ} =-y^{p^{r+s+1}}_{2},\   [y^{p^r}_{3},\    y^{p^s}_{4}]_{\circ} =y^{p^{r+s+1}}_{3},\   \\
 &[y^{p^r}_{1},\    y^{p^s}_{5}]_{\circ} =y^{p^{r+s+1}}_{1},\   [y^{p^r}_{2},\    y^{p^s}_{5}]_{\circ} =-y^{p^{r+s+1}}_{2},\   [y^{p^r}_{3},\    y^{p^s}_{5}]_{\circ} =-2y^{p^{r+s+1}}_{3},\   \\&[y^{p^r}_{1},\    y^{p^s}_{6}]_{\circ} =y^{p^{r+s+1}}_{4},\   [y^{p^r}_{2},\    y^{p^s}_{6}]_{\circ} =-y^{p^{r+s+1}}_{3},\   [y^{p^r}_{3},\    y^{p^s}_{6}]_{\circ}=0,\   \\&[y^{p^r}_{1},\    y^{p^s}_{7}]_{\circ}=-y^{p^{r+s+1}}_{8},\   [y^{p^r}_{2},\    y^{p^s}_{7}]_{\circ}=y^{p^{r+s+1}}_{4}+y^{p^{r+s+1}}_{5},\   [y^{p^r}_{3},\    y^{p^s}_{7}]_{\circ}=y^{p^{r+s+1}}_{6},\   \\
&[y^{p^r}_{1},\    y^{p^s}_{8}]_{\circ}=0,\    [y^{p^r}_{2},\    y^{p^s}_{8}]_{\circ}=y^{p^{r+s+1}}_{1},\   [y^{p^r}_{3},\    y^{p^s}_{8}]_{\circ}=y^{p^{r+s+1}}_{5},\   [y^{p^r}_{4},\    y^{p^s}_{6}]_{\circ} =-2y^{p^{r+s+1}}_{6}\\
&[y^{p^r}_{4},\    y^{p^s}_{7}]_{\circ} =-y^{p^{r+s+1}}_{7},\   [y^{p^r}_{4},\    y^{p^s}_{8}]_{\circ}= y^{p^{r+s+1}}_{8},\   [y^{p^r}_{5},\    y^{p^s}_{6}]_{\circ}= y^{p^{r+s+1}}_{6},\   \\&[y^{p^r}_{5},\    y^{p^s}_{7}]_{\circ}= -y^{p^{r+s+1}}_{7},\   [y^{p^r}_{5},\    y^{p^s}_{8}]_{\circ}= -2y^{p^{r+s+1}}_{8}.
\end{aligned}
\end{equation}

\section{Main Result and Its Proof}\label{xxsec4}
In this section, we will state and prove our  main result. Let us recall that $r \in \Omega_G$ is
normal if $r\Omega_G = \Omega_Gr$. Our purpose in this section is to study the normal elements of the completed group
algebra over  $G =\Gamma_1({\rm SL}_3(\mathbb{Z}_p))$.

For the remainder of this section, we fix the following notations: For a vector $\alpha=(\alpha_1,\alpha_2,\cdots,\alpha_n)$ of integers and any $n$-tuple  $\mathbf{y}=(y_1,y_2,\cdots,y_n)$, we write
$$\langle\alpha\rangle= \alpha_1+\alpha_2+\cdots+\alpha_n,\ \ \mathbf{y}^{\alpha}
=y_1^{\alpha_1}y_2^{\alpha_2}\cdots y_n^{\alpha_n}.$$

The first main result can now be stated:

\begin{theorem}\label{xxsec4.t1}
Let  $G =\Gamma_1({\rm SL}_3(\mathbb{Z}_p))$ and $\Omega_G$ be its completed group algebra over the
field $\mathbb{F}_p$. Then there are no nontrivial normal elements in $\Omega_G$.
\end{theorem}

\begin{proof}
Suppose that $W$ is a nontrivial normal element of $\Omega_G$ and $W$ is of the form
$$W = w_m + w_{m+1} + w_{m+2}+ \cdots+ w_d + \cdots,
$$
where $w_d (d=m, m+1, m+2, \cdots, m\geq 1)$ are homogeneous polynomials with respect
to $y_1, y_2, \cdots, y_8$ of degree $d$. That is, $w_d$ has the form
$$
 w_d =\sum_{\alpha\in \mathbb{N}^8,\ \langle\alpha\rangle=d}a_{\alpha}\mathbf{y}^{\alpha}, \ \ \ \  a_\alpha\in \mathbb{F}_p,
$$
where $\alpha=(\alpha_1,\alpha_2,\cdots,\alpha_8)\in \mathbb{N}^8$ and $\mathbf{y}^{\alpha}
=y_1^{\alpha_1}y_2^{\alpha_2}\cdots y_8^{\alpha_8}$. Moreover, we put
$$
s_d= \max\{\ s\ | \ p^s \text{ is a common divisor of the elements of each}\  \alpha \ \text{in}\ w_d, a_{\alpha}\neq 0\ \},
$$
which will be frequently invoked in the sequel.

Since $W$ is a normal element, there exists an element $\delta_k(r)\in\Omega_G$ such that
\begin{equation}\label{xxsec4.e1}
[y_k^{p^r},W]=W\cdot \delta_k(r)
\end{equation}
for each $y_k(k=1,2,\cdots, 8)$.
For a further discussion of (\ref{xxsec4.e1}),
we define
$$
s=\min\{\ s_d\ |\ d=m,m+1,m+2,\cdots \ \}.
$$
 So we get to divide the proof of the theorem into two cases: $s = s_m$ and $s < s_m$.

 \textbf{Case 1}.  $s = s_m$. In this case, by (\ref{xxsec4.e1}) we get
  \begin{equation}\label{xxsec4.e5}
[y_k^{p^r},w_m]_{\circ}=w_m\cdot (\delta_k(r))_{\circ}
\end{equation}for each $y_k(k=1,2,\cdots, 8)$. Recall that  $[y_k^{p^r},w_m]_{\circ}$ and $(\delta_k(r))_{\circ}$ stand for the  lowest
degree terms in $[y_k^{p^r},w_m]$ and $ \delta_k(r)$, respectively. It should be pointed out that  $[y_k^{p^r},w_m]_{\circ}$ is a homogeneous polynomial of degree $m-p^s+p^{r+s+1}$.

We can assume the lowest degree homogeneous polynomial $w_m$ of $W$ is of the form
\begin{equation}\label{xxsec4.e6}
\begin{aligned}
 & w_m =\sum_{i_1=0}^{\alpha_1}\cdots\sum_{i_8=0}^{\alpha_8} a_{i_1\cdots i_8}(y_1^{p^s})^{i_1}(y_2^{p^s})^{i_2}\cdots (y_8^{p^s})^{i_8}\\
&\in  \mathbb{F}_p[y_1^{p^s},y_2^{p^s},\cdots, y_8^{p^s}]\backslash \mathbb{F}_p[y_1^{p^{s+1}},y_2^{p^{s+1}},\cdots, y_8^{p^{s+1}}]
,\end{aligned}\end{equation}
where $\mathbb{F}_p[y_1^{p^s},y_2^{p^s},\cdots, y_8^{p^s}]$
denotes the polynomial ring generated by $y_1^{p^s},y_2^{p^s},\cdots, y_8^{p^s}$ over the field $\mathbb{F}_p$.
Then we can compute each $[y_k^{p^r},w_m]_{\circ}$. For $k=1$, by (\ref{xxsec3.e11}), we obtain
\begin{eqnarray}
 &[y_1^{p^{r}},w_m]_{\circ}=\{\sum_{i_1=0}^{\alpha_1}\cdots\sum_{i_8=0}^{\alpha_8} a_{i_1\cdots i_8}(y_1^{p^s})^{i_1}y_1^{p^{r}}(y_2^{p^s})^{i_2}\cdots (y_8^{p^s})^{i_8}\nonumber\\&\ \ -\sum_{i_1=0}^{\alpha_1}\cdots\sum_{i_8=0}^{\alpha_8} a_{i_1\cdots i_8}(y_1^{p^s})^{i_1}(y_2^{p^s})^{i_2}\cdots (y_8^{p^s})^{i_8}y_1^{p^{r}}\}_{\circ}\nonumber\\
&=\{ \sum_{i_1=0}^{\alpha_1}\cdots\sum_{i_8=0}^{\alpha_8} a_{i_1\cdots i_8}(y_1^{p^s})^{i_1}(y_2^{p^s})^{i_2}(y_3^{p^s})y_1^{p^{r}}(y_3^{p^s})^{i_3-1}\cdots (y_8^{p^s})^{i_8}\nonumber\\&\ \ -\sum_{i_1=0}^{\alpha_1}\cdots\sum_{i_8=0}^{\alpha_8} a_{i_1\cdots i_8}(y_1^{p^s})^{i_1}(y_2^{p^s})^{i_2}\cdots (y_8^{p^s})^{i_8}y_1^{p^{r}}\}_{\circ}\nonumber\\
&\ \ \ +\sum_{i_1=0}^{\alpha_1}\cdots\sum_{i_8=0}^{\alpha_8} a_{i_1\cdots i_8}(y_1^{p^s})^{i_1}(y_2^{p^s})^{i_2}y_2^{p^{r+s+1}}(y_3^{p^s})^{i_3-1}\cdots (y_8^{p^s})^{i_8}\nonumber\\
&= \{(\sum_{i_1=0}^{\alpha_1}\cdots\sum_{i_8=0}^{\alpha_8} a_{i_1\cdots i_8}(y_1^{p^s})^{i_1}(y_2^{p^s})^{i_2}(y_3^{2p^s})y_1^{p^{r}}(y_3^{p^s})^{i_3-2}\cdots (y_8^{p^s})^{i_8}\nonumber\\&\ \ -\sum_{i_1=0}^{\alpha_1}\cdots\sum_{i_8=0}^{\alpha_8} a_{i_1\cdots i_8}(y_1^{p^s})^{i_1}(y_2^{p^s})^{i_2}\cdots (y_8^{p^s})^{i_8}y_1^{p^{r}}\}_{\circ}\nonumber\\
&\ \ \ +\sum_{i_1=0}^{\alpha_1}\cdots\sum_{i_8=0}^{\alpha_8} a_{i_1\cdots i_8}(y_1^{p^s})^{i_1}(y_2^{p^s})^{i_2}y_2^{p^{r+s+1}}(y_3^{p^s})^{i_3-1}\cdots (y_8^{p^s})^{i_8}\nonumber\\
&\ \ \ +\sum_{i_1=0}^{\alpha_1}\cdots\sum_{i_8=0}^{\alpha_8} a_{i_1\cdots i_8}(y_1^{p^s})^{i_1}(y_2^{p^s})^{i_2}(y_3^{p^s})y_2^{p^{r+s+1}}(y_3^{p^s})^{i_3-2}\cdots (y_8^{p^s})^{i_8}\nonumber\\
&\hspace{3cm}   \vdots\nonumber\\
&= \{\sum_{i_1=0}^{\alpha_1}\cdots\sum_{i_8=0}^{\alpha_8} a_{i_1\cdots i_8}(y_1^{p^s})^{i_1}(y_2^{p^s})^{i_2}(y_3^{p^s})^{i_3}y_1^{p^{r}}\cdots (y_8^{p^s})^{i_8}\nonumber\\&\ \ -\sum_{i_1=0}^{\alpha_1}\cdots\sum_{i_8=0}^{\alpha_8} a_{i_1\cdots i_8}(y_1^{p^s})^{i_1}(y_2^{p^s})^{i_2}\cdots (y_8^{p^s})^{i_8}y_1^{p^{r}}\}_{\circ}\nonumber\\
&\ \ \ +\sum_{i_1=0}^{\alpha_1}\cdots\sum_{i_8=0}^{\alpha_8} a_{i_1\cdots i_8}(y_1^{p^s})^{i_1}(y_2^{p^s})^{i_2}y_2^{p^{r+s+1}}(y_3^{p^s})^{i_3-1}\cdots (y_8^{p^s})^{i_8}\nonumber\\
&+\sum_{i_1=0}^{\alpha_1}\cdots\sum_{i_8=0}^{\alpha_8} a_{i_1\cdots i_8}(y_1^{p^s})^{i_1}(y_2^{p^s})^{i_2}(y_3^{p^s})y_2^{p^{r+s+1}}(y_3^{p^s})^{i_3-2}\cdots (y_8^{p^s})^{i_8}+\cdots\nonumber\\
&=\{ \sum_{i_1=0}^{\alpha_1}\cdots\sum_{i_8=0}^{\alpha_8} a_{i_1\cdots i_8}(y_1^{p^s})^{i_1}(y_2^{p^s})^{i_2}(y_3^{p^s})^{i_3}y_1^{p^{r}}\cdots (y_8^{p^s})^{i_8}
\nonumber\\&\ \ -\sum_{i_1=0}^{\alpha_1}\cdots\sum_{i_8=0}^{\alpha_8} a_{i_1\cdots i_8}(y_1^{p^s})^{i_1}(y_2^{p^s})^{i_2}\cdots (y_8^{p^s})^{i_8}y_1^{p^{r}}\}_{\circ}\nonumber\\
&\ \ \ +\frac{\partial w_m}{\partial  y_3^{p^s}}y_2^{p^{r+s+1}}\nonumber\\
&= \{\sum_{i_1=0}^{\alpha_1}\cdots\sum_{i_8=0}^{\alpha_8} a_{i_1\cdots i_8}(y_1^{p^s})^{i_1}\cdots(y_4^{p^s})^{i_4}y_1^{p^{r}}\cdots (y_8^{p^s})^{i_8}\nonumber\\&\ \ -\sum_{i_1=0}^{\alpha_1}\cdots\sum_{i_8=0}^{\alpha_8} a_{i_1\cdots i_8}(y_1^{p^s})^{i_1}(y_2^{p^s})^{i_2}\cdots (y_8^{p^s})^{i_8}y_1^{p^{r}}\}_{\circ}\nonumber\\
&\ \ +\frac{\partial w_m}{\partial  y_3^{p^s}}y_2^{p^{r+s+1}}-2\frac{\partial w_m}{\partial  y_4^{p^s}}y_1^{p^{r+s+1}}\nonumber\\
&= \{\sum_{i_1=0}^{\alpha_1}\cdots\sum_{i_8=0}^{\alpha_8} a_{i_1\cdots i_8}(y_1^{p^s})^{i_1}\cdots(y_5^{p^s})^{i_5}y_1^{p^{r}}\cdots (y_8^{p^s})^{i_8}\nonumber\\&\ \ -\sum_{i_1=0}^{\alpha_1}\cdots\sum_{i_8=0}^{\alpha_8} a_{i_1\cdots i_8}(y_1^{p^s})^{i_1}(y_2^{p^s})^{i_2}\cdots (y_8^{p^s})^{i_8}y_1^{p^{r}}\}_{\circ}\nonumber\\
&\ \ +\frac{\partial w_m}{\partial  y_3^{p^s}}y_2^{p^{r+s+1}}-2\frac{\partial w_m}{\partial  y_4^{p^s}}y_1^{p^{r+s+1}}+\frac{\partial w_m}{\partial  y_5^{p^s}}y_1^{p^{r+s+1}}\nonumber\\
&=\{\sum_{i_1=0}^{\alpha_1}\cdots\sum_{i_8=0}^{\alpha_8} a_{i_1\cdots i_8}(y_1^{p^s})^{i_1}\cdots(y_6^{p^s})^{i_6}y_1^{p^{r}}(y_7^{p^s})^{i_7} (y_8^{p^s})^{i_8}\nonumber\\&\ \ -\sum_{i_1=0}^{\alpha_1}\cdots\sum_{i_8=0}^{\alpha_8} a_{i_1\cdots i_8}(y_1^{p^s})^{i_1}(y_2^{p^s})^{i_2}\cdots (y_8^{p^s})^{i_8}y_1^{p^{r}}\}_{\circ}\nonumber\\
&\ \ +\frac{\partial w_m}{\partial  y_3^{p^s}}y_2^{p^{r+s+1}}-2\frac{\partial w_m}{\partial  y_4^{p^s}}y_1^{p^{r+s+1}}+\frac{\partial w_m}{\partial  y_5^{p^s}}y_1^{p^{r+s+1}}+\frac{\partial w_m}{\partial  y_6^{p^s}}y_4^{p^{r+s+1}}\nonumber\\
&= \{\sum_{i_1=0}^{\alpha_1}\cdots\sum_{i_8=0}^{\alpha_8} a_{i_1\cdots i_8}(y_1^{p^s})^{i_1}\cdots(y_7^{p^s})^{i_7} y_1^{p^{r}}(y_8^{p^s})^{i_8}\nonumber\\&\ \ -\sum_{i_1=0}^{\alpha_1}\cdots\sum_{i_8=0}^{\alpha_8} a_{i_1\cdots i_8}(y_1^{p^s})^{i_1}(y_2^{p^s})^{i_2}\cdots (y_8^{p^s})^{i_8}y_1^{p^{r}}\}_{\circ}\nonumber\\
&\ \ +\frac{\partial w_m}{\partial  y_3^{p^s}}y_2^{p^{r+s+1}}-2\frac{\partial w_m}{\partial  y_4^{p^s}}y_1^{p^{r+s+1}}+\frac{\partial w_m}{\partial  y_5^{p^s}}y_1^{p^{r+s+1}}+\frac{\partial w_m}{\partial  y_6^{p^s}}y_4^{p^{r+s+1}}-\frac{\partial w_m}{\partial  y_7^{p^s}}y_8^{p^{r+s+1}}\nonumber\\
&=\frac{\partial w_m}{\partial  y_3^{p^s}}y_2^{p^{r+s+1}}-2\frac{\partial w_m}{\partial  y_4^{p^s}}y_1^{p^{r+s+1}}+\frac{\partial w_m}{\partial  y_5^{p^s}}y_1^{p^{r+s+1}}+\frac{\partial w_m}{\partial  y_6^{p^s}}y_4^{p^{r+s+1}}-\frac{\partial w_m}{\partial  y_7^{p^s}}y_8^{p^{r+s+1}}.\nonumber
\end{eqnarray}
Again by (\ref{xxsec3.e11}), the same argument  gives
\begin{eqnarray}
&[y_2^{p^{r}},w_m]_{\circ}=-\frac{\partial w_m}{\partial  y_4^{p^s}}y_2^{p^{r+s+1}}-\frac{\partial w_m}{\partial  y_5^{p^s}}y_2^{p^{r+s+1}}-\frac{\partial w_m}{\partial  y_6^{p^s}}y_3^{p^{r+s+1}}\nonumber\\&\qquad \qquad\qquad\qquad+\frac{\partial w_m}{\partial  y_7^{p^s}}(y_4^{p^{r+s+1}}+y_5^{p^{r+s+1}})+\frac{\partial w_m}{\partial  y_8^{p^s}}y_1^{p^{r+s+1}},\nonumber\\
&[y_3^{p^{r}},w_m]_{\circ}=-\frac{\partial w_m}{\partial  y_1^{p^s}}y_2^{p^{r+s+1}}+\frac{\partial w_m}{\partial  y_4^{p^s}}y_3^{p^{r+s+1}}-\frac{\partial w_m}{\partial  y_5^{p^s}}2y_3^{p^{r+s+1}}\nonumber\\&+\frac{\partial w_m}{\partial  y_7^{p^s}}y_6^{p^{r+s+1}}+\frac{\partial w_m}{\partial  y_8^{p^s}}y_5^{p^{r+s+1}},\nonumber\\
&[y_4^{p^{r}},w_m]_{\circ}=\frac{\partial w_m}{\partial  y_1^{p^s}}2y_1^{p^{r+s+1}}+\frac{\partial w_m}{\partial  y_2^{p^s}}y_2^{p^{r+s+1}}
-\frac{\partial w_m}{\partial  y_3^{p^s}}y_3^{p^{r+s+1}}\nonumber\\&\qquad \qquad\qquad\qquad-\frac{\partial w_m}{\partial  y_6^{p^s}}2y_6^{p^{r+s+1}}-\frac{\partial w_m}{\partial  y_7^{p^s}}y_7^{p^{r+s+1}}+\frac{\partial w_m}{\partial  y_8^{p^s}}y_8^{p^{r+s+1}},\nonumber\\
&[y_5^{p^{r}},w_m]_{\circ}=-\frac{\partial w_m}{\partial  y_1^{p^s}}y_1^{p^{r+s+1}}+\frac{\partial w_m}{\partial  y_2^{p^s}}y_2^{p^{r+s+1}}
+\frac{\partial w_m}{\partial  y_3^{p^s}}2y_3^{p^{r+s+1}}\nonumber\\&\qquad \qquad\qquad\qquad+\frac{\partial w_m}{\partial  y_6^{p^s}}y_6^{p^{r+s+1}}-\frac{\partial w_m}{\partial  y_7^{p^s}}y_7^{p^{r+s+1}}-\frac{\partial w_m}{\partial  y_8^{p^s}}2y_8^{p^{r+s+1}},\nonumber\\
&[y_6^{p^{r}},w_m]_{\circ}=-\frac{\partial w_m}{\partial  y_1^{p^s}}y_4^{p^{r+s+1}}+\frac{\partial w_m}{\partial  y_2^{p^s}}y_3^{p^{r+s+1}}
+\frac{\partial w_m}{\partial  y_4^{p^s}}2y_6^{p^{r+s+1}}\nonumber\\&-\frac{\partial w_m}{\partial  y_5^{p^s}}y_6^{p^{r+s+1}}-\frac{\partial w_m}{\partial  y_8^{p^s}}y_7^{p^{r+s+1}},\nonumber\\
&[y_7^{p^{r}},w_m]_{\circ}  =\frac{\partial w_m}{\partial  y_1^{p^s}}y_8^{p^{r+s+1}}-\frac{\partial w_m}{\partial  y_2^{p^s}}(y_4^{p^{r+s+1}}+y_5^{p^{r+s+1}})
\nonumber\\&\qquad \qquad\qquad\qquad -\frac{\partial w_m}{\partial  y_3^{p^s}}y_6^{p^{r+s+1}}+\frac{\partial w_m}{\partial  y_4^{p^s}}y_7^{p^{r+s+1}}+\frac{\partial w_m}{\partial  y_5^{p^s}}y_7^{p^{r+s+1}},\nonumber\\
&[y_8^{p^{r}},w_m]_{\circ}=-\frac{\partial w_m}{\partial  y_2^{p^s}}y_1^{p^{r+s+1}}
-\frac{\partial w_m}{\partial  y_3^{p^s}}y_5^{p^{r+s+1}}-\frac{\partial w_m}{\partial  y_4^{p^s}}y_8^{p^{r+s+1}}\nonumber\\&\qquad   +\frac{\partial w_m}{\partial  y_5^{p^s}}2y_8^{p^{r+s+1}}+\frac{\partial w_m}{\partial  y_6^{p^s}}y_7^{p^{r+s+1}}.\nonumber
\end{eqnarray}
Those identities together with (\ref{xxsec4.e5}) give rise to

\begin{equation}\label{xxsec4.e10}
\left\{
\begin{aligned}
 & \frac{\partial w_m}{\partial  y_3^{p^s}}y_2^{p^{r+s+1}}-2\frac{\partial w_m}{\partial  y_4^{p^s}}y_1^{p^{r+s+1}}+\frac{\partial w_m}{\partial  y_5^{p^s}}y_1^{p^{r+s+1}}+\frac{\partial w_m}{\partial  y_6^{p^s}}y_4^{p^{r+s+1}}\\&-\frac{\partial w_m}{\partial  y_7^{p^s}}y_8^{p^{r+s+1}}=w_m\cdot (\delta_1(r))_{\circ},\\
&{\phantom{=[y_2^{p^{r}},w_m]_{\circ}=}}\vdots\\
&-\frac{\partial w_m}{\partial  y_1^{p^s}}y_4^{p^{r+s+1}}+\frac{\partial w_m}{\partial  y_2^{p^s}}y_3^{p^{r+s+1}}
+\frac{\partial w_m}{\partial  y_4^{p^s}}2y_6^{p^{r+s+1}} -\frac{\partial w_m}{\partial  y_5^{p^s}}y_6^{p^{r+s+1}}\\&-\frac{\partial w_m}{\partial  y_8^{p^s}}y_7^{p^{r+s+1}}=w_m\cdot (\delta_6(r))_{\circ},\\
&{\phantom{=[y_2^{p^{r}},w_m]_{\circ}=}}\vdots
\end{aligned}
\right.
\end{equation}

\begin{claim}\label{xxsec4.c1}
$\frac{\partial w_m}{\partial  y_k^{p^s}}(k=1,2,\cdots,8)$ are not exactly all zeros.
\end{claim}
\begin{proof}  In view of (\ref{xxsec4.e6}) we can rewrite $w_m$ as
$$w_m=\sum_{{i_1}=0}^{\alpha_1}(y_1^{p^s})^{i_1}v_{i_1}(y_2^{p^s}, y_3^{p^s}, \cdots,  y_8^{p^s}),$$
where $v_{i_1}(y_2^{p^s}, y_3^{p^s}, \cdots,  y_8^{p^s})\in\mathbb{F}_p[y_2^{p^s},y_3^{p^s},\cdots, y_8^{p^s}]$.
Suppose that on the contradictory the claim, then we have
\begin{equation}\label{xxsec4.ec1cha30}\frac{\partial w_m}{\partial  y_1^{p^s}}=i_1\sum_{i_1=1}^{\alpha_1}(y_1^{p^s})^{i_1-1}v_{i_1}(y_2^{p^s}, y_3^{p^s}, \cdots,  y_8^{p^s})=0.\end{equation}
 To begin with, let us look on (\ref{xxsec4.ec1cha30}) as a polynomial related to $  y_1^{p^s}$. Then for each $i_1(i_1=1,2,\cdots,\alpha_1)$,
$${i_1}v_{i_1}(y_2^{p^s}, y_3^{p^s}, \cdots,  y_8^{p^s})=0.$$
This implies that $i_1=b_{i_1}p^{n_{i_1}}(i_i=1,2,\cdots,\alpha_1)$, where $n_{i_1}\geq1$, $\gcd(b_{i_1},p)=1$. Similarly, we can rewrite $w_m$ as
\begin{eqnarray}
&\label{xxsec4.ec1cha32}w_m=\sum_{{i_2}=0}^{\alpha_2}(y_2^{p^s})^{i_2}v_{i_2}(y_1^{p^s}, y_3^{p^s}, \cdots,  y_8^{p^s}), \\
&\ \ \ \ \ \ \ \ \ \ \ \ \ \ \ \ \vdots\nonumber\\
&\label{xxsec4.ec1cha33}w_m=\sum_{{i_8}=0}^{\alpha_8}(y_8^{p^s})^{i_8}v_{i_8}[(y_1^{p^s}, y_2^{p^s}, \cdots,  y_7^{p^s}),
\end{eqnarray}
respectively, where $$  \begin{aligned}&v_{i_2}(y_1^{p^s},  y_3^{p^s}, \cdots,  y_8^{p^s})\in\mathbb{F}_p[y_1^{p^s},y_3^{p^s},\cdots, y_8^{p^s}],\\
&\ \ \ \ \ \ \ \ \ \ \ \ \ \ \ \ \vdots\\
&v_{i_8}(y_1^{p^s}, y_2^{p^s}, \cdots,  y_7^{p^s})\in\mathbb{F}_p[y_1^{p^s},y_2^{p^s},\cdots, y_7^{p^s}]. \end{aligned} $$ Repeating the
above analogous proof we arrive at  $i_2=b_{i_2}p^{n_{i_2}}(i_2=1,2,\cdots,\alpha_2)$, where $n_{i_2}\geq1$, $\gcd(b_{i_2},p)=1$, $\cdots$,
$i_8=b_{i_8}p^{n_{i_8}}(i_8=1,2,\cdots,\alpha_8)$, where $n_{i_8}\geq1$, $\gcd(b_{i_8},p)=1$. This shows that
$$w_m\in\mathbb{F}_p[y_1^{p^{s+1}},y_2^{p^{s+1}},\cdots, y_8^{p^{s+1}}],$$
which is contradictory to the assumption
$$w_m\in\mathbb{F}_p[y_1^{p^s},y_2^{p^s},\cdots, y_8^{p^s}]\backslash \mathbb{F}_p[y_1^{p^{s+1}},y_2^{p^{s+1}},\cdots, y_8^{p^{s+1}}].$$

\end{proof}

Let us now come back to the system of equations (\ref{xxsec4.e10}). By Claim \ref{xxsec4.c1} and without loss of generality,
we may assume that $\frac{\partial w_m}{\partial  y_1^{p^s}}\neq 0$. Then there exists one positive integer $r\gg0$ such that
$$\begin{aligned}&-\frac{\partial w_m}{\partial  y_1^{p^s}}y_4^{p^{r+s+1}}+\frac{\partial w_m}{\partial  y_2^{p^s}}y_3^{p^{r+s+1}}
+\frac{\partial w_m}{\partial  y_4^{p^s}}2y_6^{p^{r+s+1}}-\frac{\partial w_m}{\partial  y_5^{p^s}}y_6^{p^{r+s+1}}-\frac{\partial w_m}{\partial  y_8^{p^s}}y_7^{p^{r+s+1}}\\&=w_m\cdot (\delta_6(r))_{\circ}\neq0,\end{aligned}$$
which can be rearranged as
\begin{equation}\label{xxsec4.e20}\begin{aligned}&-\frac{\partial w_m}{\partial  y_1^{p^s}}y_4^{p^{r+s+1}}+\frac{\partial w_m}{\partial  y_2^{p^s}}y_3^{p^{r+s+1}}
+\frac{\partial w_m}{\partial  y_4^{p^s}}2y_6^{p^{r+s+1}}-\frac{\partial w_m}{\partial  y_5^{p^s}}y_6^{p^{r+s+1}}-\frac{\partial w_m}{\partial  y_8^{p^s}}y_7^{p^{r+s+1}}\\&=w_m \sum_{i=3,4,6,7}U_6^i(y_1, y_2, \cdots,  y_8)y_i^{p^r} \neq0,\end{aligned}\end{equation}
where $ \sum_{i=3,4,6,7}U_6^i(y_1, y_2, \cdots,  y_8)y_i^{p^r}= (\delta_6(r))_{\circ}$.   Comparing the coefficients of $y_4^{p^r}$ of the above relation, we further get
 \begin{equation}\label{xxsec4.ebu20} -\frac{\partial w_m}{\partial  y_1^{p^s}}y_4^{p^{r+s+1}-p^r}=w_m U_6^4(y_1, y_2, \cdots,  y_8)  \neq0. \end{equation}
 Taking into account (\ref{xxsec4.ec1cha33}) and comparing the degree of $ y_8^{p^s}$ in the two sides of (\ref{xxsec4.ebu20}),  we obtain
$$ \begin{aligned}&-( y_8^{p^s})^{\alpha_8}\frac{\partial  v_{i_8}(y_1^{p^s}, \cdots, y_7^{p^s})}{\partial  y_1^{p^s}}y_4^{p^{r+s+1}-p^r}
\\& =( y_8^{p^s})^{\alpha_8}v_{i_8}(y_1^{p^s}, \cdots,  y_7^{p^s})g(y_1, y_2, \cdots,  y_7) \neq0.
\end{aligned} $$
where $g(y_1, y_2, \cdots,  y_7)$ stands for the sum of certain terms in $U_6^4(y_1, y_2, \cdots,  y_8)$. It follows that
$$ -\frac{\partial  v_{i_8}(y_1^{p^s}, \cdots, y_7^{p^s})}{\partial  y_1^{p^s}}y_4^{p^{r+s+1}-p^r}
  = v_{i_8}(y_1^{p^s}, \cdots,  y_7^{p^s})g(y_1, y_2, \cdots,  y_7) \neq0.$$
Comparing the degree of $y_1^{p^s}$ in the two sides of the above equality, we immediately arrive at a contradiction. This implies
that $W$ is not a nontrivial normal element of $\Omega_G$ under the case of $s=s_m$.

\textbf{Case 2}.  $s < s_m$. Now there exists some fixed $d$ with $d>m$ such that $s=s_d<s_m$, and it follows from  (\ref{xxsec4.e1}) that
  \begin{equation}\label{xxsec4.c2e1}
[y_k^{p^r},w_d]_{\circ}=w_m\cdot (\delta_k(r))_{\circ}
\end{equation} for each $y_k(k=1,2,\cdots, 8)$  provided $r\gg 0$. To proceed our discussion, we assume that $w_d$ is of the form
\begin{equation}\label{xxsec4.c2e2}
\begin{aligned}
 & w_d =\sum_{i_1=0}^{p-1}\cdots\sum_{i_8=0}^{p-1} (y_1^{p^s})^{i_1}(y_2^{p^s})^{i_2} \cdots (y_8^{p^s})^{i_8} h_{i_1i_2 \cdots i_8}(y_1^{p^{s+1}},y_2^{p^{s+1}},\cdots, y_8^{p^{s+1}})\\
&\in  \mathbb{F}_p[y_1^{p^s},y_2^{p^s},\cdots, y_8^{p^s}]\backslash \mathbb{F}_p[y_1^{p^{s+1}},y_2^{p^{s+1}},\cdots, y_8^{p^{s+1}}] ,
\end{aligned}\end{equation}
where $\mathbb{F}_p[y_1^{p^s},y_2^{p^s},\cdots, y_8^{p^s}]$
 denotes the polynomial ring generated by $y_1^{p^s},y_2^{p^s},\cdots, y_8^{p^s}$ over the field $\mathbb{F}_p$.

Using the computational method of Case 1 and producing a system of partial differential equations:
\begin{eqnarray}
& \label{xxsec4.s2e1}\frac{\partial w_d}{\partial  y_3^{p^s}}y_2^{p^{r+s+1}}-2\frac{\partial w_d}{\partial  y_4^{p^s}}y_1^{p^{r+s+1}}+\frac{\partial w_d}{\partial  y_5^{p^s}}y_1^{p^{r+s+1}}+\frac{\partial w_d}{\partial  y_6^{p^s}}y_4^{p^{r+s+1}}\\&-\frac{\partial w_d}{\partial  y_7^{p^s}}y_8^{p^{r+s+1}}=w_m\cdot (\delta_1(r))_{\circ},\nonumber\\
&\label{xxsec4.s2e2}-\frac{\partial w_d}{\partial  y_4^{p^s}}y_2^{p^{r+s+1}}-\frac{\partial w_d}{\partial  y_5^{p^s}}y_2^{p^{r+s+1}}-\frac{\partial w_d}{\partial  y_6^{p^s}}y_3^{p^{r+s+1}}+  \frac{\partial w_d}{\partial  y_7^{p^s}}(y_4^{p^{r+s+1}}+y_5^{p^{r+s+1}})\\&+\frac{\partial w_d}{\partial  y_8^{p^s}}y_1^{p^{r+s+1}}=w_m\cdot (\delta_2(r))_{\circ},\nonumber\\
&\label{xxsec4.s2e3}-\frac{\partial w_d}{\partial  y_1^{p^s}}y_2^{p^{r+s+1}}+\frac{\partial w_d}{\partial  y_4^{p^s}}y_3^{p^{r+s+1}}-\frac{\partial w_d}{\partial  y_5^{p^s}}2y_3^{p^{r+s+1}}
+\frac{\partial w_d}{\partial  y_7^{p^s}}y_6^{p^{r+s+1}}\\&+\frac{\partial w_d}{\partial  y_8^{p^s}}y_5^{p^{r+s+1}}=w_m\cdot (\delta_3(r))_{\circ}\nonumber,\\
&\label{xxsec4.s2e4}\frac{\partial w_d}{\partial  y_1^{p^s}}2y_1^{p^{r+s+1}}+\frac{\partial w_d}{\partial  y_2^{p^s}}y_2^{p^{r+s+1}}
-\frac{\partial w_d}{\partial  y_3^{p^s}}y_3^{p^{r+s+1}}-\frac{\partial w_d}{\partial  y_6^{p^s}}2y_6^{p^{r+s+1}}\\& -\frac{\partial w_d}{\partial  y_7^{p^s}}y_7^{p^{r+s+1}}+\frac{\partial w_d}{\partial  y_8^{p^s}}y_8^{p^{r+s+1}}=w_m\cdot (\delta_4(r))_{\circ}\nonumber,\\
&\label{xxsec4.s2e5}-\frac{\partial w_d}{\partial  y_1^{p^s}}y_1^{p^{r+s+1}}+\frac{\partial w_d}{\partial  y_2^{p^s}}y_2^{p^{r+s+1}}
+\frac{\partial w_d}{\partial  y_3^{p^s}}2y_3^{p^{r+s+1}}+\frac{\partial w_d}{\partial  y_6^{p^s}}y_6^{p^{r+s+1}} \\&-\frac{\partial w_d}{\partial  y_7^{p^s}}y_7^{p^{r+s+1}}-\frac{\partial w_d}{\partial  y_8^{p^s}}2y_8^{p^{r+s+1}}=w_m\cdot (\delta_5(r))_{\circ}\nonumber,\\
&\label{xxsec4.s2e6}-\frac{\partial w_d}{\partial  y_1^{p^s}}y_4^{p^{r+s+1}}+\frac{\partial w_d}{\partial  y_2^{p^s}}y_3^{p^{r+s+1}}
+\frac{\partial w_d}{\partial  y_4^{p^s}}2y_6^{p^{r+s+1}} -\frac{\partial w_d}{\partial  y_5^{p^s}}y_6^{p^{r+s+1}}\\&-\frac{\partial w_d}{\partial  y_8^{p^s}}y_7^{p^{r+s+1}}=w_m\cdot (\delta_6(r))_{\circ}\nonumber,\\
&\label{xxsec4.s2e7}\frac{\partial w_d}{\partial  y_1^{p^s}}y_8^{p^{r+s+1}}-\frac{\partial w_d}{\partial  y_2^{p^s}}(y_4^{p^{r+s+1}}+y_5^{p^{r+s+1}})
-\frac{\partial w_d}{\partial  y_3^{p^s}}y_6^{p^{r+s+1}} +\frac{\partial w_d}{\partial  y_4^{p^s}}y_7^{p^{r+s+1}}\\&+\frac{\partial w_d}{\partial  y_5^{p^s}}y_7^{p^{r+s+1}}=w_m\cdot (\delta_7(r))_{\circ}\nonumber,\\
&\label{xxsec4.s2e8}-\frac{\partial w_d}{\partial  y_2^{p^s}}y_1^{p^{r+s+1}}
-\frac{\partial w_d}{\partial  y_3^{p^s}}y_5^{p^{r+s+1}}-\frac{\partial w_d}{\partial  y_4^{p^s}}y_8^{p^{r+s+1}} +\frac{\partial w_d}{\partial  y_5^{p^s}}2y_8^{p^{r+s+1}}\\&+\frac{\partial w_d}{\partial  y_6^{p^s}}y_7^{p^{r+s+1}}=w_m\cdot (\delta_8(r))_{\circ}\nonumber.
\end{eqnarray}

Before continuing our proof,  we need to state two further claims, which are established in below:

\begin{claim}\label{xxsec4.c5}For each $\frac{\partial w_d}{\partial  y_k^{p^s}}(k=1,2,\cdots,8)$,
there exist the following divisible relations $w_m|\frac{\partial w_d}{\partial  y_k^{p^s}}(k=1,2,\cdots,8)$.
\end{claim}

\begin{proof}Suppose that on the contradictory  $w_m\nmid\frac{\partial w_d}{\partial  y_1^{p^s}}$.
Let us choose a positive integer $r\gg0$. On the one hand, by (\ref{xxsec4.s2e3}) we know that
$$
\begin{aligned}
&-\frac{\partial w_d}{\partial  y_1^{p^s}}y_2^{p^{r+s+1}}+\frac{\partial w_d}{\partial  y_4^{p^s}}y_3^{p^{r+s+1}}-\frac{\partial w_d}{\partial  y_5^{p^s}}2y_3^{p^{r+s+1}}
+\frac{\partial w_d}{\partial  y_7^{p^s}}y_6^{p^{r+s+1}}+\frac{\partial w_d}{\partial  y_8^{p^s}}y_5^{p^{r+s+1}}\\&=w_m \sum_{i=2,3,5,6}U_3^i(y_1, y_2, \cdots,  y_8)y_i^{p^{r}},
\end{aligned}
$$
where $ \sum_{i=2,3,5,6}U_3^i(y_1, y_2, \cdots,  y_8)y_i^{p^{r}}= (\delta_3(r))_{\circ}$. On the other hand, using (\ref{xxsec4.s2e6}), we have
$$
\begin{aligned}
&-\frac{\partial w_d}{\partial  y_1^{p^s}}y_4^{p^{r+s+1}}+\frac{\partial w_d}{\partial  y_2^{p^s}}y_3^{p^{r+s+1}}
+\frac{\partial w_d}{\partial  y_4^{p^s}}2y_6^{p^{r+s+1}} -\frac{\partial w_d}{\partial  y_5^{p^s}}y_6^{p^{r+s+1}} -\frac{\partial w_d}{\partial  y_8^{p^s}}y_7^{p^{r+s+1}}\\&=w_m \sum_{i=3,4,6,7}U_6^i(y_1, y_2, \cdots,  y_8)y_i^{p^{r}},
\end{aligned}
$$
where $\sum_{i=3,4,6,7}U_6^i(y_1, y_2, \cdots,  y_8)y_i^{p^{r}}=(\delta_6(r))_{\circ}$. Comparing the coefficients of $y_2^{p^{r}}$ and $y_4^{p^{r}}$ in the two sides of the above two equalities, we infer that
$w_m|\frac{\partial w_d}{\partial  y_1^{p^s}}y_2^{p^{r+s+1}-p^r}$ and $w_m|\frac{\partial w_d}{\partial  y_1^{p^s}}y_4^{p^{r+s+1}-p^r}$, respectively.
Note that our hypothesis implies that there exists an irreducible polynomial $f$ such that
$f^{\beta}\nmid \frac{\partial w_d}{\partial  y_1^{p^s}}$ with  $f^{\beta}| w_m$ for some positive integer $\beta$, so $f$ is a common divisor of
$y_2^{{p^{r+s+1}-p^r}}$ and $y_4^{{p^{r+s+1}-p^r}}$, which is a contradiction. Thus $w_m|\frac{\partial w_d}{\partial  y_1^{p^s}}$.
A similar argument shows that  $w_m|\frac{\partial w_d}{\partial  y_k^{p^s}}(k=2, 3, 6, 7, 8)$.

Now (\ref{xxsec4.s2e1})-(\ref{xxsec4.s2e8}) degenerate into
\begin{eqnarray}
 & \label{xxsec4.s2e9}-2\frac{\partial w_d}{\partial  y_4^{p^s}}y_1^{p^{r+s+1}}+\frac{\partial w_d}{\partial  y_5^{p^s}}y_1^{p^{r+s+1}}=w_m\cdot (\delta_1(r))_{\circ} , \\
&\label{xxsec4.s2e10}-\frac{\partial w_d}{\partial  y_4^{p^s}}y_2^{p^{r+s+1}}-\frac{\partial w_d}{\partial  y_5^{p^s}}y_2^{p^{r+s+1}}=w_m\cdot (\delta_2(r))_{\circ} , \\
&\label{xxsec4.s2e11}\frac{\partial w_d}{\partial  y_4^{p^s}}y_3^{p^{r+s+1}}-\frac{\partial w_d}{\partial  y_5^{p^s}}2y_3^{p^{r+s+1}}=w_m\cdot (\delta_3(r))_{\circ} , \\
&\label{xxsec4.s2e12}\frac{\partial w_d}{\partial  y_4^{p^s}}2y_6^{p^{r+s+1}}-\frac{\partial w_d}{\partial  y_5^{p^s}}y_6^{p^{r+s+1}}=w_m\cdot (\delta_6(r))_{\circ} , \\
&\label{xxsec4.s2e13}\frac{\partial w_d}{\partial  y_4^{p^s}}y_7^{p^{r+s+1}}+\frac{\partial w_d}{\partial  y_5^{p^s}}y_7^{p^{r+s+1}}=w_m\cdot (\delta_7(r))_{\circ} , \\
&\label{xxsec4.s2e14}-\frac{\partial w_d}{\partial  y_4^{p^s}}y_8^{p^{r+s+1}}+\frac{\partial w_d}{\partial  y_5^{p^s}}2y_8^{p^{r+s+1}}=w_m\cdot (\delta_8(r))_{\circ} ,
\end{eqnarray}
respectively. If $p\neq3$, it is a easy to verify  $w_m|(\frac{\partial w_d}{\partial  y_4^{p^s}}+\frac{\partial w_d}{\partial  y_5^{p^s}})$
and $ w_m|(2\frac{\partial w_d}{\partial  y_4^{p^s}}-\frac{\partial w_d}{\partial  y_5^{p^s}})$. And hence
$w_m|\frac{\partial w_d}{\partial  y_4^{p^s}}$ and $w_m|\frac{\partial w_d}{\partial  y_5^{p^s}}$, as desired.
As to the case of $p=3$, the results are actually the same, the further discussion are omitted here.
\end{proof}

\begin{claim}\label{xxsec4.c6}For $w_m, w_d$, there exist $u \in\mathbb{F}_p[y_1^{p^s}
, y_2^{p^s}, \cdots, y_8^{p^s} ]$ and $
v \in\mathbb{F}_p[y_1^{p^{s+1}}$
, $y_2^{p^{s+1}}$,  $\cdots$,  $y_8^{p^{s+1}}]$ such that $w_d=w_mu+v$.
\end{claim}

\begin{proof}
Claim \ref{xxsec4.c5} tells us that there exist  $u_k\in \mathbb{F}_p[y_1^{p^s}
, y_2^{p^s}, \cdots, y_8^{p^s} ]$ such that $\frac{\partial w_d}{\partial  y_k^{p^s}}=w_mu_k(k=1,2,\cdots,8)$. Let $u_k(k=1,2,\cdots,8)$ be of the following
forms:
$$
 u_k =\sum_{ i_1=0}^{p-1}\sum_{ i_2=0}^{p-1}\cdots\sum_{  i_8=0}^{p-1}(y_1^{p^s})^{i_1}(y_2^{p^s})^{i_2}\cdots (y_8^{p^s})^{ i_8}g_{i_1i_2\cdots i_8}^k(y_1^{p^{s+1}},y_2^{p^{s+1}},\cdots, y_8^{p^{s+1}}).
$$
Then for $k=1$, we have
 $$\begin{aligned}
\frac{\partial w_d}{\partial  y_1^{p^s}}&=w_m\sum_{ i_1=0}^{p-1}\sum_{ i_2=0}^{p-1}\cdots\sum_{  i_8=0}^{p-1}(y_1^{p^s})^{i_1}(y_2^{p^s})^{i_2}\cdots (y_8^{p^s})^{ i_8}\\& \ \ \ \ \times g_{i_1i_2\cdots i_8}^1(y_1^{p^{s+1}},y_2^{p^{s+1}},\cdots, y_8^{p^{s+1}}).
\end{aligned}
$$
On the other hand, it follows from (\ref{xxsec4.c2e2}) that
 $$\begin{aligned}
\frac{\partial w_d}{\partial  y_1^{p^s}}&=\sum_{ i_1=1}^{p-1}\sum_{ i_2=0}^{p-1}\cdots\sum_{  i_8=0}^{p-1}i_1(y_1^{p^s})^{i_1-1}(y_2^{p^s})^{i_2}\cdots (y_8^{p^s})^{ i_8}\\& \ \ \ \ \times h_{i_1i_2\cdots i_8}(y_1^{p^{s+1}},y_2^{p^{s+1}},\cdots, y_8^{p^{s+1}})\\
&=\sum_{ i_1=0}^{p-2}\sum_{ i_2=0}^{p-1}\cdots\sum_{  i_8=0}^{p-1}(i_1+1)(y_1^{p^s})^{i_1}(y_2^{p^s})^{i_2}\cdots (y_8^{p^s})^{ i_8}\\& \ \ \ \ \times h_{(i_1+1)i_2\cdots i_8}(y_1^{p^{s+1}},y_2^{p^{s+1}},\cdots, y_8^{p^{s+1}}).
\end{aligned}
$$
Comparing the last two relations, we see that
$$\begin{aligned}
&w_mg_{i_1i_2\cdots i_8}^1(y_1^{p^{s+1}},y_2^{p^{s+1}},\cdots, y_8^{p^{s+1}})\\
&=(i_1+1)h_{(i_1+1)i_2\cdots i_8}(y_1^{p^{s+1}},y_2^{p^{s+1}},\cdots, y_8^{p^{s+1}}).
\end{aligned}
$$
Similarly, we also get
\begin{eqnarray}
&w_mg_{i_1i_2\cdots i_8}^2(y_1^{p^{s+1}},y_2^{p^{s+1}},\cdots, y_8^{p^{s+1}})\nonumber\\
&=(i_2+1)h_{i_1(i_2+1)\cdots i_8}(y_1^{p^{s+1}},y_2^{p^{s+1}},\cdots, y_8^{p^{s+1}})\nonumber,\\
&\qquad\qquad\qquad\vdots\nonumber\\
&w_mg_{i_1i_2\cdots i_8}^8(y_1^{p^{s+1}},y_2^{p^{s+1}},\cdots, y_8^{p^{s+1}})\nonumber\\
&=( i_8+1)h_{i_1i_2\cdots( i_8+1)}(y_1^{p^{s+1}},y_2^{p^{s+1}},\cdots, y_8^{p^{s+1}})\nonumber.
\end{eqnarray}
This shows that
$$w_m|h_{i_1i_2\cdots  i_8 }(y_1^{p^{s+1}},y_2^{p^{s+1}},\cdots, y_8^{p^{s+1}}),$$ where $i_1, i_2, \cdots, i_8$ are not complete zeroes. That is, for each $h_{i_1i_2\cdots  i_8 }(y_1^{p^{s+1}}$, $y_2^{p^{s+1}}$, $\cdots$, $y_8^{p^{s+1}})$, there exists a corresponding $h_{i_1i_2\cdots i_8}^{*}$ such that
\begin{equation}\label{xxsec4.c2e9}h_{i_1i_2\cdots i_8}(y_1^{p^{s+1}},y_2^{p^{s+1}},\cdots, y_8^{p^{s+1}})=w_mh_{i_1i_2\cdots i_8}^{*},\ i_1\geq1  \ \text{or} \  i_2\geq1\cdots \ \text{or} \  i_8\geq1.\end{equation}
 Taking   (\ref{xxsec4.c2e9}) into (\ref{xxsec4.c2e2}) yields
$$\begin{aligned}
&w_d=w_m\sum_{ i_1=0}^{p-1}\sum_{ i_2=0}^{p-1}\cdots\sum_{  i_8=0}^{p-1}(y_1^{p^s})^{i_1}(y_2^{p^s})^{i_2}\cdots (y_8^{p^s})^{ i_8}h_{i_1i_2\cdots i_8}^{*}\\&{\phantom{w_d=}}+h_{00\cdots0} (y_1^{p^{s+1}},y_2^{p^{s+1}},\cdots, y_8^{p^{s+1}}),\end{aligned}
$$
where $i_1, i_2, \cdots, i_8$ are not all zero.
Let us write
$$\begin{aligned}&u=\sum_{ i_1=0}^{p-1}\sum_{ i_2=0}^{p-1}\cdots\sum_{  i_8=0}^{p-1}(y_1^{p^s})^{i_1}(y_2^{p^s})^{i_2}\cdots (y_8^{p^s})^{ i_8}h_{i_1i_2\cdots i_8}^{*}, \\&  v= h_{00\cdots0} (y_1^{p^{s+1}},y_2^{p^{s+1}},\cdots, y_8^{p^{s+1}}),\end{aligned}$$
where $i_1, i_2, \cdots, i_8$ are not all zero.   Then  $w_d=w_mu+v$, where $u \in\mathbb{F}_p[y_1^{p^s}$
, $y_2^{p^s}$, $\cdots$, $y_8^{p^s} ]$ and $
v \in\mathbb{F}_p[y_1^{p^{s+1}}
, y_2^{p^{s+1}},  \cdots ,  y_8^{p^{s+1}} ]$. The result follows.
\end{proof}

We now continue to proceed our proof. Let us consider the following set of $\Omega_G$ as
$$\begin{aligned}N(w_m) =& \{\ W\ |\ W \ \text{is a nontrivial normal element  }\\&   \text{with the lowest degree term}\   w_m, s(W)= s_m-1\ \},
\end{aligned}
$$
 where $s(W)$ is the $s$ corresponding to
$W$.
For any $W \in N(w_m)$, we assume that $s(W)=s_d$  for some $d > m$. Thus one can
write $W$ as
$$W = w_m + w_{m+1} + w_{m+2}+ \cdots+ w_d + \cdots. $$
Then by Claim \ref{xxsec4.c6} we have $w_d=w_mu+v$,
 where $$u \in\mathbb{F}_p[y_1^{p^{ s_m-1}}
, y_2^{p^{ s_m-1}}, \cdots, y_8^{p^{ s_m-1}} ],\
v \in\mathbb{F}_p[y_1^{p^{ s_m}}
, y_2^{p^{ s_m}},  \cdots ,  y_8^{p^{ s_m}} ].$$ For convenience, we denote the index of $w_d$ by $d(W)$.
Let us write $W = W_0$ and $W_1 = W(1 - u)$. Then
$$\begin{aligned}
&W_1= w_m + w_{m+1} + w_{m+2} + \cdots + (w_d  - w_mu) + (w_{d+1}  - w_{m+1}u)\\
&\ \ \ \ \ \ +(w_{d+2}- w_{m+2}u) + \cdots  + (w_{2d-m}- w_du) + \cdots.
\end{aligned}
$$
It is easy to verfiy that $W_1 \in N(w_m)$ and $d(W_0) < d(W_1)$.
Likewise, for $ W_1$,
there exist $u'$ and $v'$ such that $w_{1d} = w_mu'+v'$, where
$w_{1d}$ is the first homogeneous
polynomial satisfying the condition  $s(W_1)= s_m-1$ in $W_1$. We set $W_2 = W_1(1 - u')$. It
is also easy to check that
$W_2 \in N(w_m)$ and $d(W_1) < d(W_2)$. Repeating this
process continuously, we finally construct an infinite sequence of normal elements
$$W_0 = W, \ W_1 = W(1 - u),  \  W_2 = W(1 - u)(1 - u'),\  \cdots.$$
Let us set ${\rm lim}_{n\rightarrow \infty}W_n = V$. Then $V$  is a normal element with the form
$$V =  v_m + v_{m+1} +\cdots  v_{d-1} + v_{d}+ \cdots ,$$
where $v_m=w_m$. It follows that $s(V)> s_{m-1}$, a contradiction. This shows that $W$ is not a nontrivial normal  element of  $\Omega_G$ under  the case of $s<s_m$.
\end{proof}

\begin{remark}
We would like to point out that the current computational method can be used to
discuss the normal elements of the completed group algebra $\Omega_G$ over
$G={\Gamma_1(\rm SL}_2(\mathbb{Z}_p))$. Conversely, the adopted method of \cite{WeiBian1}
can not be adapted to the current situation. One distinguished difference can be
observed by comparing the proof of Claim \ref{xxsec4.c5} with that of Claim 11 of \cite{WeiBian1}.
\end{remark}

\section{Topics for Further Research}
\label{xxsec5}

As you known, the main purpose of the current article is to study normal elements of
a completed group algebra over the special linear group ${\rm SL}_3(\mathbb{Z}_p)$.
Those analogous questions on completed group algebras defined over other $p$-adic
groups also have great interest and draw more people's attention.
In this section, we will present several potential topics for future further research.
Motivated by our current work, Clozel's systematic work \cite{Clozel1, Clozel2, Clozel3}
and  Ray's papers \cite{Ray1, Ray2} , it is natural to propose several questions in this line.

For a few small $p$, there are some extra difficulties and challenges to compute normal elements of
completed group algebras over ${\rm SL}_n(\mathbb{Z}_p)$. For example,
in the case of $p=2$, $G=\Gamma_1({\rm SL}_2(\mathbb{Z}_p))$ will have $p$-torsion and thus
its completed group algebra is not an integral domain which prevents one
from using deep results of Lazard \cite{Lazard}. Although we exclude these
primes from consideration in the stage, we strongly believe that we should say
much more about the normal elements and ideals of the completed group algebra $\Omega_G$.

\begin{question}\label{xxsec5.1}
Let $G=\Gamma_1({\rm SL}_n(\mathbb{Z}_2))$ be the first congruence kernel of ${\rm SL}_n(\mathbb{Z}_2)$
and $\Omega_G$ be its completed group algebra over $\mathbb{F}_p$.
Are there any non-trivial normal elements $\Omega_G$ ?
\end{question}

One much more common question is as the following:

\begin{question}\label{xxsec5.2}
Let $G=\Gamma_1({\rm SL}_n(\mathbb{Z}_p))(n>3)$ and $\Omega_G$ be its
completed group algebra over ${\mathbb{F}}_p$. Are there any non-trivial normal
elements in $\Omega_G$ ?
\end{question}

Question \ref{xxsec5.2} will involve rather complicated and tedious computations. In particular, when $p$ is a divisor
of $n$, we have not found a reasonable approach to this question.

Let $G$ be a semi-simple, simply connected Chevalley group over $\mathbb{Z}_p$
and $G(\mathbb{Z}_p)$ be its $\mathbb{Z}_p$-points. Under a faithful
representation of group schemes $\rho: G\hookrightarrow {\rm GL}_n$ over $\mathbb{Z}$,
one can define, for each $k\in \mathbb{N}$, $\Gamma(k):={\rm ker}({\rm GL}_n(\mathbb{Z}_p)\longrightarrow  {\rm GL}_n(\mathbb{Z}_p/p^k\mathbb{Z}_p))$
(the $\mathbb{Z}$-structure on ${\rm GL}_n$ being given by $V_{\mathbb Z}$) and $G(k):=G(\mathbb{Z}_p)\cap \Gamma(k)$.
Then $G(k)$ is called the \textit{$k$-th congruence kernel} of $G(\mathbb{Z}_p)$ which satisfies a
descending filtration $G(1)\supseteq G(2) \supseteq G(3) \supseteq \cdots $.
Ray \cite{Ray1} give an explicit presentation (by generators and relations) of the completed group algebra
for the first congruence kernel of a semi-simple, simply connected Chevalley group over $\mathbb{Z}_p$,
extending the proof given by Clozel for the group $\Gamma_1({\rm SL}_2(\mathbb{Z}_p))$,
the first congruence kernel of ${\rm SL}_2(\mathbb{Z}_p)$ for primes $p>2$. This
immediately gives rise to the following question.

\begin{question}\label{xxsec5.3}
Let $G$ be a semi-simple, simply connected Chevalley group over $\mathbb{Z}_p$, $G(1)$ be
the first congruence kernel of $G(\mathbb{Z}_p)$ and $\Omega_{G(1)}$ be its
completed group algebra over ${\mathbb{F}}_p$. Are there any non-trivial normal
elements in $\Omega_{G(1)}$ ?
\end{question}

For a prime $p>n+1$, Ray \cite{Ray2} determine explicitly the presentation in the form
of generators and relations of the completed group algebras $\Lambda_G$ and $\Omega_G$ over the pro-$p$
Iwahori subgroup $G$ of ${\rm GL}_n(\mathbb{Z}_p)$. Let $G$ be the pro-$p$ Iwahori
subgroup of ${\rm GL}_n(\mathbb{Z}_p)$,  i.e. $G$ is the group of matrices in ${\rm GL}_n(\mathbb{Z}_p)$
which are upper unipotent modulo the maximal ideal $p\mathbb{Z}_p$ of $\mathbb{Z}_p$. It is
natural to form the following conjecture.

\begin{question}\label{xxsec5.4}
Let $G$ be the pro-$p$ Iwahori subgroup of ${\rm GL}_n(\mathbb{Z}_p)$ and $\Omega_G$ be its
completed group algebra over ${\mathbb{F}}_p$. Are there any non-trivial normal elements
in $\Omega_G$ ?
\end{question}

\end{document}